\documentclass{amsart}

\usepackage{macros}

\title{Banach Spaces from Barriers in High Dimensional
                                             Ellentuck Spaces}

\author{A. Arias}
\address{Department of Mathematics\\
  University of Denver \\
  C.M. Knudson Hall, Room 300
2390 S. York St.\\ Denver, CO \ 80208 U.S.A.}  
\email{aarias@du.edu}
\urladdr{http://cs.du.edu/~aarias}

\author{N.
  Dobrinen}
\address{Department of Mathematics\\
  University of Denver \\
   C.M. Knudson Hall, Room 300
2390 S. York St. \\ Denver, CO \ 80208 U.S.A.}  
\email{natasha.dobrinen@du.edu}
\urladdr{\url{http://web.cs.du.edu/~ndobrine}} 
\thanks{Dobrinen was partially  supported by  National Science Foundation Grants DMS-1301665 and DMS-1600781}

\author{G. Gir\'{o}n-Garnica}
\address{Department of Mathematics\\
  University of Denver \\
  C.M. Knudson Hall, Room 300
2390 S. York St.\\ Denver, CO \ 80208 U.S.A.}  
\email{ggirngar@du.edu}
\urladdr{http://www.gabrielgiron.com}

\author{J. G. Mijares}
\address{Department of Mathematical and Statistical Sciences\\
 University of Colorado Denver  \\
   Student Commons Building 
1201 Larimer St. \\ Denver, CO \ 80217 U.S.A.}  
\email{jose.mijarespalacios@ucdenver.edu}
\urladdr{http://math.ucdenver.edu/~mijaresj/}

\begin{document} 


\begin{abstract}
A new   hierarchy of  Banach spaces $T_k(d,\theta)$, $k$ any positive integer, is constructed using barriers in high dimensional Ellentuck spaces \cite{DobrinenJSL15}
following the classical framework under which a Tsirelson type norm is defined from a barrier in the
Ellentuck space \cite{Argyros/TodorcevicBK}. 
The following structural properties of these spaces are proved.
Each of  these spaces contains arbitrarily large
copies of $\ell_\infty^n$, with the bound constant for all $n$.
For each fixed pair $d$ and $\theta$,
the spaces $T_k(d,\theta)$, $k\ge 1$,
are $\ell_p$-saturated, 
forming natural extensions of the $\ell_p$ space,
where $p$ satisfies $d\theta=d^{1/p}$.
Moreover,
they  form  a strict hierarchy over the $\ell_p$ space:
For any $j<k$,
the space $T_j(d,\theta)$ embeds isometrically into  
$T_k(d,\theta)$ as a subspace which is non-isomorphic to $T_k(d,\theta)$.
\end{abstract}

\maketitle

\section{Introduction}

Banach space theory is rich with applications  of   fronts and barriers within the framework of the
Ellentuck space (see for example \cite{Diestel}, \cite{Odell}, and Part B of \cite{Argyros/TodorcevicBK}). 
Infinite-dimensional Ramsey theory is a branch of Ramsey Theory initiated by Nash-Williams in the
course of developing his theory of better-quasi-ordered sets in the early 1960's. This theory
introduced the notions of fronts and barriers that turned out to be  important in the context
of Tsirelson type norms. 
During the 1970's, Nash-Williams' theory was reformulated and strengthened
by the work of 
Silver \cite{Silver71},
Galvin and Prikry \cite{Galvin/Prikry73},  Louveau \cite{Louveau74},  and Mathias  \cite{Mathias77},   culminating in  Ellentuck's work in \cite{Ellentuck74} introducing topological
Ramsey theory on what  is now called the Ellentuck space.

Building on  work of Carlson and Simpson in \cite{Carlson/Simpson90}, Todorcevic  distilled key properties of the Ellentuck space into four axioms which
guarantee that a topological space satisfies infinite-dimensional Ramsey theory  analogously to \cite{Ellentuck74}
(see Chapter 5 of \cite{TodorcevicBK10}).
These abstractions of the Ellentuck space are called topological Ramsey spaces.
In particular,  Todorcevic
has shown
that the theory of fronts and barriers in the Ellentuck space  extends to the
context of general topological Ramsey spaces.
This general theory has already found applications finding exact  initial segments of the 
Tukey structure of ultrafilters  in \cite{Raghavan/Todorcevic12}, \cite{Dobrinen/Todorcevic14}, \cite{Dobrinen/Todorcevic15}, 
\cite{Dobrinen/Mijares/Trujillo14}, and
\cite{DobrinenJSL15}.
In this context,
 the second author constructed a new hierarchy of 
topological Ramsey spaces 
in \cite{DobrinenJSL15} and \cite{DobrinenIDE15}
which form dense subsets of 
 the  Boolean algebras 
$\mathcal{P}(\om^{\al})/\Fin^{\al}$,
for $\al$ any countable ordinal.
Those constructions were motivated by the following.

The
Boolean algebra $\mathcal{P}(\om)/\Fin$,  the Ellentuck space, and Ramsey ultrafilters are closely connected.
A Ramsey ultrafilter is the strongest type of ultrafilter,
 satisfying the following partition relation: 
For each partition of the pairs of natural numbers into two pieces, there is a member of the ultrafilter such that all pairsets coming from that member are in the same piece of the partition.
Ramsey ultrafilters can be constructed via standard methods using $\mathcal{P}(\om)/\Fin$ and the Continuum Hypothesis or other set-theoretic techniques such as forcing.
The Boolean algebra $\mathcal{P}(\om^2)/\Fin^2$ is the next step in complexity above $\mathcal{P}(\om)/\Fin$.
This Boolean algebra can be used to generate
an ultrafilter $\mathcal{U}_2$ on base set $\om\times\om$  satisfying a weaker   partition relation: 
For each partition of the pairsets on $\om$ into five or more pieces, there is a member of $\mathcal{U}_2$ such that the pairsets on that member are all contained in four pieces of the partition.
Moreover, the projection of $\mathcal{U}_2$ to the first copy of $\om$ recovers a Ramsey ultrafilter.
In \cite{Blass/Dobrinen/Raghavan15}, many aspects of the ultrafilter $\mathcal{U}_2$ were investigated, but the exact  structure of the Tukey, equivalently cofinal, types below $\mathcal{U}_2$ remained open.
The topological Ramsey space $\mathcal{E}_2$ and more generally the spaces $\mathcal{E}_k$, $k\ge 2$, were constructed to produce  dense subsets of $\mathcal{P}(\om^k)/\Fin^k$
which form  topological Ramsey spaces,
thus setting the stage for 
 finding the exact structure  of the cofinal types of  all ultrafilters
Tukey reducible  to the ultrafilter  generated   by $\mathcal{P}(\om^k)/\Fin^k$  in  \cite{DobrinenJSL15}.

Once constructed, it became clear that  these new spaces  are the natural generalizations of the Ellentuck space to higher dimensions, and hence are called high-dimensional Ellentuck spaces.
This,  in conjunction with the multitude of  results on  Banach spaces  constructed using barriers on the original Ellentuck space, 
led the second author to infer  that  the general theory of barriers on these high-dimensional Ellentuck spaces would be a natural starting point for answering the following question.

\begin{question}\label{q.D}
What Banach spaces can be constructed by extending Tsirelson's construction method by using barriers in general topological Ramsey spaces?
\end{question}

The constructions presented in this paper have as their starting point Tsirelson's groundbreaking example of a
reflexive Banach space $T$ with an unconditional basis not containing $c_0$ or $\ell_p$ with $1
\leq p < \infty$ \cite{Tsirelson}. 
The idea of Tsirelson's construction became apparent
after Figiel and Johnson \cite{FigielJohnson} showed that the norm of the Tsirelson space satisfies the
following equation:
\vspace{4pt}
\beq \label{eqTsirelsonNorm}
\norm{\sum\nolimits_n a_n e_n} = \max \left \{ \sup\nolimits_n \abs{a_n}, \frac{1}{2} \sup
\sum_{i=1}^{m} \norm{E_i \left( \sum\nolimits_n a_n e_n \right)} \right \},
\eeq

\noindent where the $\sup$ is taken over all sequences $(E_i)_{i=1}^m$ of successive finite subsets
of integers with the property that $m \leq \min(E_1)$ and $E_i \left( \sum_n a_n e_n
\right) = \sum_{n \in E_i} a_n e_n$.

The first systematic abstract study of Tsirelson's construction was achieved by Argyros and Deliyanni
\cite{ArgyrosDeliyanni1}. Their construction starts with a real number $0 < \theta < 1$ and an
arbitrary family $\mcF$ of finite subsets of $\mbN$ that is the downwards closure of a barrier in
Ellentuck space. Then, one defines the \textit{Tsirelson type space} $T(\mcF,\theta)$ as the
completion of $\czz(\mbN)$ with the implicitly given norm (\ref{eqTsirelsonNorm}) replacing $1/2$
by $\theta$ and using sequences $(E_i)_{i=1}^m$ of finite subsets of positive integers which are
$\mcF$-admissible, i.e., there is some $\set{k_1,k_2,\ldots,k_m} \in \mcF$ such that
$k_1 \leq \min(E_1) \leq \max(E_1) < k_2 \leq \cdots < k_m \leq \min(E_m) \leq \max(E_m).$

In this notation, Tsirelson's original space is denoted by $T(\mcS,1/2)$, where $\mcS = \set{F
\subset \mbN : \abs{F} \leq \min(F)}$ is the Schreier family. 
In addition to  $\mcS$, the
\textit{low complexity hierarchy} $\{\mcA_d\}_{d=1}^\infty$ with
\vspace{4pt}
\beqs
\mcA_d := \set{F \subset \mbN : \abs{F} \leq d}
\eeqs

\noindent is noteworthy in the realm of Tsirelson type spaces. In fact, Bellenot proved
in \cite{Bellenot} the following remarkable theorem:

\bthm[Bellenot \cite{Bellenot}] \label{thmLCH}
If $d \theta > 1$, then for every $x \in T(\mcA_d, \theta)$,
\vspace{4pt}
\beqs
\frac{1}{2d} \norm{x}_p \leq \norm{x}_{T(\mcA_d, \theta)} \leq \norm{x}_p,
\eeqs

\noindent where $d\theta = d^{1/p}$ and $\norm{\cdot}_p$ denotes the $\ell_p$-norm.
\ethm

We  construct new Banach spaces  extending the low complexity hierarchy on the Ellentuck space  to  the low complexity hierarchy on  the
finite dimensional Ellentuck spaces.
This produces various  structured extensions of the $\ell_p$ spaces.
In Section \ref{secHDES} we review the construction of the finite-dimensional Ellentuck spaces in  \cite{DobrinenJSL15}, introducing   
new notation and  representations more suitable 
to the context of this paper.
Our  new Banach spaces $T_k(d,\theta)$  are constructed in Section \ref{sec.Bconstruction}, using finite rank barriers on the $k$-dimensional Ellentuck spaces.
These spaces   may be thought of as structured  generalizations of $\ell_p$, where 
 $d\theta=d^{1/p}$,
as they extends
the construction of
the Tsirelson type space $T_1(d,\theta)$,
 which by Bellenot's Theorem \ref{thmLCH} is  exactly  $\ell_p$.
We prove the following structural results about the spaces $T_k(d,\theta)$.
The spaces $T_k(d,\theta)$ 
contain arbitrarily large copies of $\ell_\infty^n$, where the bound is fixed for all $n$ (Section \ref{sec.subspaceslNinfty}),
and that 
there are many natural  block subspaces isomorphic to
$\ell_p$ (Section \ref{sec.blocksubspaces}). 
Moreover,  they are $\ell_p$-saturated (Section \ref{sec.l_psat}).
The spaces $T_k(d,\theta)$ are 
not isomorphic to each other (Section \ref{sec.notiso}),
but 
   for each $j<k$,  there are subspaces of $T_k(d,\theta)$ 
which are isometric to 
 $T_j(d,\theta)$ (Section \ref{sec.embedding}).
Thus, for fixed $d,\theta$,
the spaces $T_k(d,\theta)$, $k\ge 1$, form a natural hierarchy in complexity over $\ell_p$, where $d\theta=d^{1/p}$.

As there are several different ways to naturally generalize the Tsirelson construction to high dimensional Ellentuck spaces,
 we consider a  second, more stringent definition of norm and construct a second type of space $T(\mathcal{A}^k_d,\theta)$, where the admissible sets are required to be separated by sets which are finite approximations to members of the $\mathcal{E}_k$ (Section \ref{sec.9}).
The norms on these spaces are thus bounded by the norms on the $T_k(d,\theta)$ spaces.
The spaces $T(\mathcal{A}^k_d,\theta)$  are shown to have 
 the same properties as the  as shown for $T_k(d,\theta)$,  the only exception being  that we do not know whether $T(\mathcal{A}^j_d,\theta)$ embeds as an isometric subspace of $T(\mathcal{A}^k_d,\theta)$
for $j<k$.
The paper concludes with  open problems for further  research into the properties of these spaces.


\section{High Dimensional Ellentuck Spaces} \label{secHDES}

In \cite{DobrinenJSL15}, the second author constructed a new hierarchy $(\mcE_k)_{2
\le k < \omega}$ of
topological Ramsey spaces  which
generalize the Ellentuck space in a natural manner.
In this section, we reproduce the construction, though with slightly different notation more suited to the context of Banach spaces.

 Recall that the Ellentuck space \cite{Ellentuck74}  is the
triple $([\om]^{\om},\subseteq,r)$, where the finitization map $r:\om\times [\om]^{\om}\rightarrow [\om]^{<\om}$ is  defined as follows:  
For each $X
\in[\om]^{\om}$ and $n<\om$,   $r(n,X)$ is the set of   the least $n$ elements of $X$. 
Usually $r(n,X)$ is denoted by $r_n(X)$.
We shall let
$\mathcal{E}_1$ denote the Ellentuck space.

We now begin the process of defining the  high dimensional Ellentuck spaces $\mathcal{E}_k$, $k\ge 2$. The presentation
here is slightly different than, but equivalent to, the one in \cite{DobrinenJSL15}. We have chosen to do so in order to
simplify the construction of the Banach spaces.
In logic, 
the set of natural numbers $\{0,1,2,\dots\}$ is denoted by 
the symbol $\om$.
In keeping with the logic influence in \cite{DobrinenJSL15},
we shall use this notation.
 We start by defining a  well-ordering on the collection of all 
non-decreasing sequences  of members of
$\om$ which forms the backbone for the structure of the members in the spaces.

\begin{defn} \label{defOmle}
For $k \ge 2$, denote by $\omle{k}$ the collection of  all non-decreasing sequences of members
of $\om$ of length less than or equal to $k$.
\end{defn}

\begin{defn}[The well-order $\lex$] \label{defLex}
Let $(s_1,\dots,s_i)$ and $(t_1,\dots,t_j)$, with $i,j \ge 1$, be in $\omle{k}$. We say that
$(s_1,\dots,s_i)$ is lexicographically below $(t_1,\dots,t_j)$, written $(s_1,\dots,s_i) \lex
(t_1,\dots,t_j)$, if and only if there is a non-negative integer $m$ with the following properties:
\begin{enumerate}
\item[(i)]   $m \leq i$ and $m \leq j$;
\item[(ii)]  for every positive integer $n \leq m, s_n = t_n$; and
\item[(iii)] either $s_{m+1} < t_{m+1}$, or $m = i$ and $m < j$.
\end{enumerate}

This is just a generalization of the way the alphabetical order of words is based on the
alphabetical order of their component letters.
\end{defn}

\bexa
Consider the sequences $(1,2), (2)$, and $(2,2)$ in $\omle{2}$. Following the preceding definition
we have $(1,2) \lex (2) \lex (2,2)$. 
We conclude that $(1,2) \lex
(2)$ by setting $m = 0$ in Definition \ref{defLex}; similarly, $(2) \lex (2,2)$ follows by setting
$m = 1$, as any proper initial segment of a sequence is lexicographically below that
sequence.
\eexa

\begin{defn}[The well-ordered set $(\omle{k},\prec)$] \label{defPrec}
Set the empty sequence $()$ to be the $\prec$-minimum element of $\omle{k}$; so, for all nonempty
sequences $s$ in $\omle{k}$, we have $() \prec s$. In general, given  $(s_1,\dots,s_i)$ and
$(t_1,\dots,t_j)$ in $\omle{k}$ with $i,j \ge 1$, define $(s_1,\dots,s_i) \prec (t_1,\dots,t_j)$
if and only if either
\begin{enumerate}
\item
$s_i < t_j$, or
\item
$s_i = t_j$ and
$(s_1,\dots,s_i) <_{\mathrm{lex}} (t_1,\dots,t_j)$.
\end{enumerate}
\end{defn}


\noindent \textbf{Notation.}
Since  $\prec$  well-orders  $\omle{k}$ in order-type $\om$, we fix the notation of letting
$\vec{s}_m$  denote the $m$-th member of $(\omle{k},\prec)$.
Let $\ome{k}$ denote the collection of all non-decreasing sequences of length $k$ of members of
$\om$. Note that $\prec$ also well-orders $\ome{k}$ in order type $\om$. Fix the notation of letting
$\vec{u}_n$ denote the $n$-th member of $(\ome{k},\prec)$.
For $s,t\in \omle{k}$,
we say that $s$ is  an {\em  initial segment} of $t$ and write $s\sqsubset t$ if 
$s=(s_1,\dots,s_i)$, $t=(t_1,\dots,t_j)$,
$i<j$,
and for all $m\le i$, $s_m=t_m$.

\begin{defn}[The spaces $(\mathcal{E}_k,\le,r)$, $k\ge 2$, Dobrinen \cite{DobrinenJSL15}] \label{defE_k}
An {\em $\mathcal{E}_k$-tree} 
 is a function   $\widehat{X}$ from $\omle{k}$ into $\omle{k}$ that preserves the well-order $\prec$  and  initial segments $\sqsubset$.
For $\widehat{X}$ an $\mathcal{E}_k$-tree, let $X$ denote the restriction of $\widehat{X}$ to
$\ome{k}$. 
The space $\mathcal{E}_k$  is defined to be the collection of all $X$ such that  $\widehat{X}$
is an $\mathcal{E}_k$-tree.
We identify $X$ with its range and usually will write $X = \set{v_1,v_2,\ldots}$, where
$v_1 = X(\vec{u}_1) \prec v_2 = X(\vec{u}_2) \prec \cdots$. 
The partial ordering on $\mathcal{E}_k$ is defined to be simply inclusion; that is, given $X,Y\in \mathcal{E}_k$,
$X\le Y$ if and only if  (the range of) $X$ is a subset of  (the range of) $Y$.
For each $n<\om$, the $n$-th restriction function $r_n$ on $\mathcal{E}_k$  is defined 
 by 
 $r_n(X) = \set{v_1,v_2,\ldots,v_n}$
that is, 
the  $\prec$-least $n$  members of 
$X$. 
When necessary for clarity, we write $r^k_n(X)$ to highlight that $X$ is a member of $\mathcal{E}_k$.
We set
\vspace{3pt}
\beqs
\mathcal{AR}_n^k := \{ r_n(X) : X \in \mathcal{E}_k \}
\hspace{1.1cm} \textnormal{and} \hspace{1.1cm}
\mathcal{AR}^k := \{ r_n(X) : n < \om, X \in \mathcal{E}_k \}
\eeqs
to denote the set of all 
$n$-th approximations to members of $\mathcal{E}_k$, 
and the set of all
finite approximations to members of $\mathcal{E}_k$, respectively.
\end{defn}

\begin{rem}
Let $s \in \omle{k}$ and denote its length  by $|s|$.
Since $\widehat{X}$ preserves initial segments, it follows that $|\widehat{X}(s)| = \abs{s}$.
Thus, $\mathcal{E}_k$ is the space of all functions $X$ from $\ome{k}$ into $\ome{k}$ which induce
an $\mathcal{E}_k$-tree. 
Notice  that the
identity function  is identified with $\ome{k}$. The set $\ome{k}$
is 
  the prototype for all members of $\mcE_k$
in the sense that every member  $X$ of $\mcE_k$ will be a subset of $\ome{k}$ which 
has the same structure
as $\ome{k}$, according to the interaction between the two orders $\prec$ and $\sqsubset$.
Definition \ref{defE_k} essentially is  generalizing  the key points about 
the structure, according to $\prec$ and $\sqsubset$,  of the
identity function on $\omle{k}$. 
\end{rem}

The family of all non-empty finite subsets of $\ome{k}$ will be denoted by $\fin(\ome{k})$. Clearly,
$\mcAR^k \subset \fin(\ome{k})$. If $E \in \fin(\ome{k})$, then we denote the minimal and maximal
elements of $E$ with respect to $\prec$ by $\minprec(E)$ and $\maxprec(E)$, respectively.

\begin{exa}[The space $\mathcal{E}_2$]\label{exE_2}
The members of $\mathcal{E}_2$ look like $\om$ many copies of the Ellentuck space.
The well-order $(\omle{2}, \prec)$  begins as follows:
\vspace{4pt}
\begin{equation*}
()\prec (0)\prec (0,0)\prec (0,1)\prec (1)\prec (1,1)\prec (0,2)\prec (1,2)\prec (2)\prec (2,2)\prec\cdots
\end{equation*}
The tree structure of $\omle{2}$, under lexicographic order, looks like $\om$
copies of $\om$, and has order type the countable ordinal $\om\cdot\om$.
Here, we picture the finite tree $\{ \vec{s}_m: 1 \le m \le 21 \}$, which indicates how the rest
of the tree $\omle{2}$ is formed. This is exactly  the tree formed by
taking all initial segments of the set $\{ \vec{u}_n : 1 \le n \le 15 \}$.

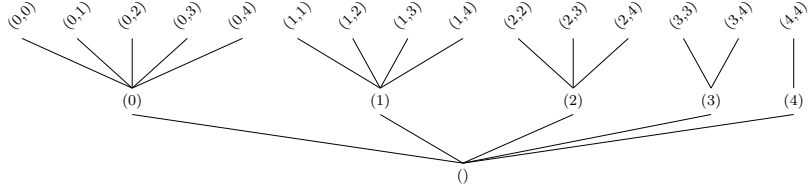
\begin{figure}[H]
{\footnotesize
\begin{tikzpicture}[scale=.71,grow'=up, level distance=40pt,sibling distance=.2cm]
\tikzset{grow'=up}
\Tree [.$()$ [.$(0)$ [.$\rotatebox{45}{(0,0)}$ ][.$\rotatebox{45}{(0,1)}$ ][.$\rotatebox{45}{(0,2)}$ ][.$\rotatebox{45}{(0,3)}$ ]  [.$\rotatebox{45}{(0,4)}$ ]   ] [.$(1)$ [.$\rotatebox{45}{(1,1)}$ ][.$\rotatebox{45}{(1,2)}$ ][.$\rotatebox{45}{(1,3)}$ ][.$\rotatebox{45}{(1,4)}$ ] ] [.$(2)$ [.$\rotatebox{45}{(2,2)}$ ] [.$\rotatebox{45}{(2,3)}$ ][.$\rotatebox{45}{(2,4)}$ ] ] [.$(3)$ [.$\rotatebox{45}{(3,3)}$ ] [.$\rotatebox{45}{(3,4)}$ ] ]  [.$(4)$  [.$\rotatebox{45}{(4,4)}$ ] ] ]
\end{tikzpicture}
}
\captionsetup{singlelinecheck=on}
\vspace{-0.1cm}
\caption{Initial structure of $\omle{2}$}
\label{figTRw2}
\end{figure}
\end{exa}

Next we present the specifics of the structure of the space $\mathcal{E}_3$.

\begin{exa}[The space $\mathcal{E}_3$]\label{exE_3}
The well-order $(\omle{3},\prec)$ begins as follows:
\vspace{4pt}
\begin{align*}
() &\prec (0) \prec (0,0)\prec (0,0,0)\prec (0,0,1)\prec (0,1)\prec (0,1,1)\prec (1)\cr
&\prec (1,1)
\prec (1,1,1)\prec
(0,0,2)\prec (0,1,2)\prec (0,2)\prec 
(0,2,2)
\cr
&\prec (1,1,2)\prec (1,2)\prec (1,2,2)\prec (2)\prec (2,2)\prec (2,2,2)\prec (0,0,3)\prec \cdots
\end{align*}

The set $\omle{3}$ is a tree of height three with each non-maximal node branching into $\om$ many nodes.  
The maximal nodes in the following figure are technically the set $\{ \vec{u}_n : 1 \le n \le 20 \}$,
which indicates  the structure of $\omle{3}$.

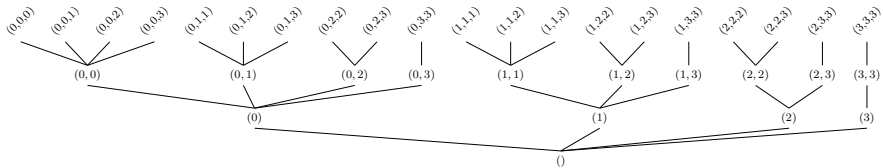
\begin{figure}[H]
{\footnotesize
\begin{tikzpicture}[scale=.54,grow'=up, level distance=30pt,sibling distance=.1cm]
\tikzset{grow'=up}
\Tree [.$()$ [.$(0)$ [.$(0,0)$ [.$\rotatebox{45}{(0,0,0)}$ ][.$\rotatebox{45}{(0,0,1)}$ ][.$\rotatebox{45}{(0,0,2)}$ ] [.$\rotatebox{45}{(0,0,3)}$ ] ][.$(0,1)$ [.$\rotatebox{45}{(0,1,1)}$ ][.$\rotatebox{45}{(0,1,2)}$ ] [.$\rotatebox{45}{(0,1,3)}$ ]][.$(0,2)$ [.$\rotatebox{45}{(0,2,2)}$ ]  [.$\rotatebox{45}{(0,2,3)}$ ]]  [.$(0,3)$ [.$\rotatebox{45}{(0,3,3)}$ ] ]   ][.$(1)$ [.$(1,1)$ [.$\rotatebox{45}{(1,1,1)}$ ][.$\rotatebox{45}{(1,1,2)}$ ] [.$\rotatebox{45}{(1,1,3)}$ ] ] [.$(1,2)$ [.$\rotatebox{45}{(1,2,2)}$ ]  [.$\rotatebox{45}{(1,2,3)}$ ]] [.$(1,3)$ [.$\rotatebox{45}{(1,3,3)}$ ] ] ][.$(2)$ [.$(2,2)$ [.$\rotatebox{45}{(2,2,2)}$ ] [.$\rotatebox{45}{(2,2,3)}$ ] ] [.$(2,3)$ [.$\rotatebox{45}{(2,3,3)}$ ] ] ][.$(3)$  [.$(3,3)$ [.$\rotatebox{45}{(3,3,3)}$ ] ] ]]
\end{tikzpicture}
}
\captionsetup{singlelinecheck=on}
\vspace{-0.1cm} 
\caption{Initial structure of $\omle{3}$}
\end{figure}
\end{exa}

\subsection{Upper Triangular Representation of $\ome{2}$}
We now present an alternative and very useful way to visualize elements of $\mcE_2$. This turned
out to be fundamental to developing more intuition and to understanding the Banach spaces that we define
in the following section. We refer to it as the \textit{upper triangular representation of}
$\ome{2}$:


The well-order $(\ome{2}, \prec)$  begins as follows: $(0,0) \prec (0,1) \prec (1,1) \prec (0,2)
\prec (1,2) \prec (2,2) \prec \cdots$. In comparison with the tree representation shown in Figure
\ref{figTRw2}, the upper triangular representation makes it simpler to visualize this well-order:
Starting at $(0,0)$ we move from top to bottom throughout each column, and then to the right to the
next column.

The following figure shows the initial part of an $\mcE_2$-tree $\widehat{X}$. The highlighted
pieces represent the restriction of $\widehat{X}$ to $\ome{2}$.
\begin{figure}[H]
\includegraphics[scale=0.3]{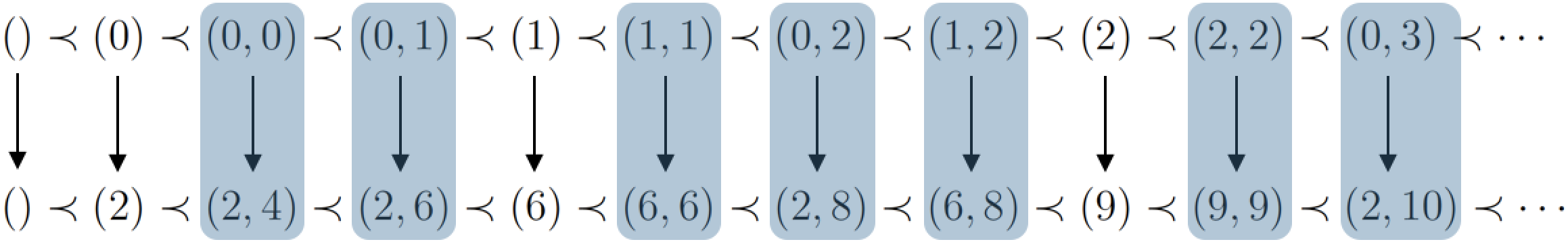}
\captionsetup{singlelinecheck=on}
\caption{Initial part of an $\mcE_2$-tree}
\end{figure}

Under the identification discussed after Definition \ref{defE_k}, we have that 
\vspace{3pt}
\beqs
X = \set{(2,4),(2,6),(6,6),(2,8),(6,8),(9,9),(2,10),\ldots}
\eeqs

\noindent is an element of $\mcE_2$. Using the upper triangular representation of $\ome{2}$ we can
visualize $r_{10}(X)$:


\subsection{Special Maximal Elements of $\mcE_k$}
There are special elements in $\mcE_k$ that are useful for describing the structure of some subspaces
of the Banach spaces that we define in the following section. Given $v \in \ome{k}$ we want to
construct a special $X_v^{\max} \in \mcE_k$ that has $v$ as its first element and that is maximal in the sense that every other $Y\in\mathcal{E}_2$ with $v$ as its $\prec$-minimum member is a subset of $X$.

For example, if $v = (0,4) \in \ome{2}$, then a finite approximation of $X_v^{\max} \in \mcE_2$
looks like this:

\begin{figure}[H]
\includegraphics[scale=0.23]{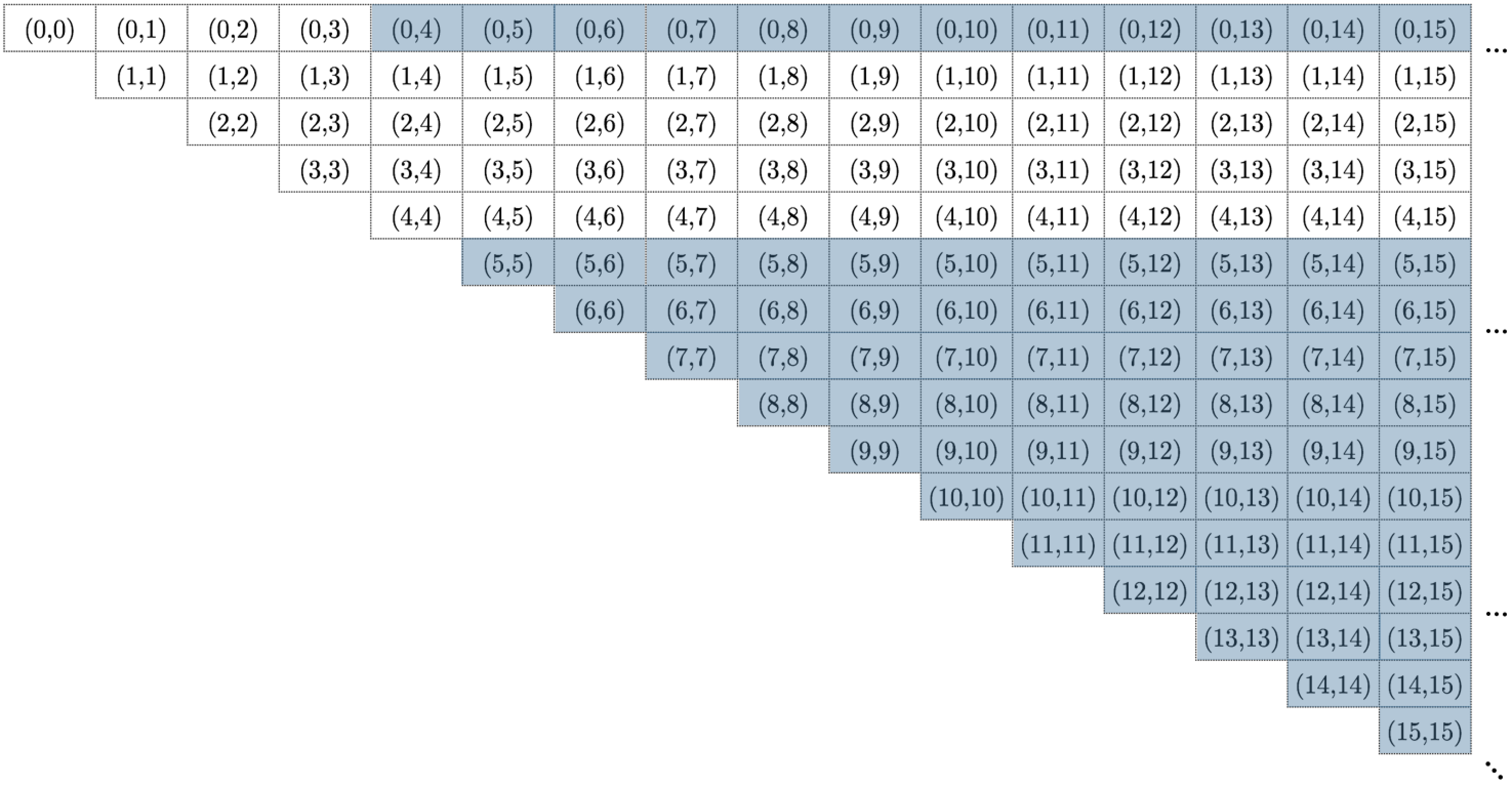}
\centering
\vspace{-1cm}
\caption{A finite approximation of $X_{(0,4)}^{\max} \in \mcE_2$}
\end{figure}

Now let us illustrate this with $k=3$ and $v=(0,2,7)$. Since we want $v$ as the first element of
$X_v^{\max}$, we identify it with $(0,0,0)$ and then we choose the next elements as small as
possible following Definition \ref{defE_k}:
\vspace{3pt}
\begin{align*}
() &\prec (0) \prec (0,2) \prec (0,2,7) \prec (0,2,8) \prec (0,8) \prec (0,8,8) \prec (8) \cr
   &\prec (8,8) \prec (8,8,8) \prec (0,2,9) \prec (0,8,9) \prec (0,9) \prec (0,9,9) \cr
   &\prec (8,8,9) \prec (8,9) \prec (8,9,9) \prec (9) \prec (9,9) \prec (9,9,9) \prec (0,2,10) \prec
          \cdots
\end{align*}

Therefore, under the identification discussed after Definition \ref{defE_k}, we have
\vspace{3pt}
\beqs
X_v^{\max} = \{(0,2,7), (0,2,8), (0,8,8), (8,8,8), (0,2,9), \ldots\}.
\eeqs

In general, for any $k\ge 2$, $X_v^{\max}$ is constructed as follows.

\begin{defn}\label{defn.Xmax}
Let $k\ge 2$ be given 
and suppose $v = (n_1,n_2,\ldots,n_k)$. First we
define the $\mcE_k$-tree $\widehat{X}_v$ that will determine $X_v^{\max}$. 
$\widehat{X}_v$ must be
a function from $\omle{k}$ to $\omle{k}$ satisfying Definition \ref{defE_k} and such that
$\widehat{X}_v(0,0,\ldots,0) = v$. So, for $m,j \in \Zp, j \le k$, define the
following auxiliary functions: $f_j(0) := n_j$ and $f_j(m) := n_k+m$. Then, for
$t = (m_1,m_2,\ldots,m_l) \in \omle{k}$ set $\widehat{X}_v(t) := (f_1(m_1),f_2(m_2),\ldots,
f_l(m_l))$. Finally, define $X_v^{\max}$ to be the restriction of $\widehat{X}_v$ to $\ome{k}$.
\end{defn}

It is routine to  check that:

\blem \label{lemElemXvmax} 
Let $v = (n_1,\ldots,n_k)$.  Then $\widehat{X}_v$ is an $\mcE_k$-tree, and $w = (m_1,\ldots,m_k)\in \ome{k}$
belongs to $X_v^{\max}$ if and only if
either  $m_1 > n_k$ or
else there is $1 \leq i \le k$ such that $(m_1,m_2,
\ldots,m_i) = (n_1,n_2,\ldots,n_i)$, and if $i<k$ then  $m_{i+1} > n_k$. 
\elem

We use this Lemma to prove that $X_v^{\max}$ contains all elements of $\mcAR^k$ that have $v$ as initial value.

\bprop \label {propXvmax}
Let $E\in \mcAR^k$ and $v=\min_{\prec}(E)$. Then $E\subset X_v^{\max}$.
\eprop

\begin{proof}
Let $E\in \mcAR^k$ with $v=\min_{\prec}(E)$.  Then there exists an $\mcE_k$-tree $\widehat{X}$ such that $\widehat{X}(0,\dots,0)=(n_1,\dots,n_k)=v$ and $E=r_j(X)$ for some $j$.
If $E$ has more than one member, then its 
 second element is
$\widehat{X}(0,\dots,0,1)=(n_1,\dots,n_{k-1}, n_k')$, 
for some $n_k'>n_k$.  By Lemma \ref{lemElemXvmax}, $\widehat{X}(0,\dots,0,1)\in X_v^{\max}$.

Suppose that $w=(p_1,\dots,p_k)$ is any member of $E$ 
besides $v$.
We will show that $w\in X_v^{\max}$. 
Let  $(m_1,\dots,m_k)\in \ome{k}$  be the sequence such that $w=\widehat{X}(m_1,\dots,m_k)$.
 Suppose first that $m_1>0$.  Then
$$(0,\dots,0,1)\prec (m_1) \prec (m_1,m_2,\dots,m_k).$$
Applying $\widehat{X}$ and recalling that $\widehat{X}$ preserves $\prec$ and $\sqsubset$, we conclude that
$$(n_1,\dots,n_{k-1},n_k')\prec (p_1)\prec (p_1,p_2,\dots,p_k).$$
Comparing the first two elements we see that
either  $p_1>n_k'$ or else $p_1=n_k'$ and $p_1>n_1$.  
In both cases, $p_1\geq n_k'>n_k$, which, by Lemma \ref{lemElemXvmax}, gives that $(p_1,p_2,\dots,p_k)=w\in X_v^{\max}$.

Suppose now that $m_1=\dots=m_i=0$ and $m_{i+1}>0$, where $i+1\le k$.  Then
$$(0,\dots,0,1)\prec (m_1,\dots,m_i,m_{i+1})=(0,\dots,0,m_{i+1})\prec (0,\dots,0,m_{i+1},\dots,m_k).$$
Applying $\widehat{X}$ we obtain 
$$(n_1,\dots,n_{k-1},n_k')\prec (n_1,\dots,n_i,p_{i+1})\prec (n_1,\dots, n_i,p_{i+1},\dots,p_k).$$
Arguing as before, we conclude that $p_{i+1}\geq n_k'>n_k$, and then Lemma \ref{lemElemXvmax} yields that $w\in X_v^{\max}$.
\end{proof}

\bcor \label{cor1LemElemXvmax}
If $w \in X_v^{\max}$, then $X_w^{\max} \subseteq X_v^{\max}$. 
In particular, if $E \in \mcAR^k$ and
$\minprec(E) \in X_v^{\max}$, then $E \subset X_v^{\max}$.
\ecor

Notice that  the only elements in  $\widehat{X}_v$ that are $\prec$-smaller than $v$ are the
initial segments of $v$.

\bcor \label{cor2LemElemXvmax}
If $s \prec v$ and $s \in \ran{\widehat{X}_v}$ then $s \sqsubset v$.
\ecor


\section{The Banach Space $T_k(d, \theta)$}\label{sec.Bconstruction}

\subsection{Preliminary Definitions}
Set $\mbN := \set{0,1,\ldots}$ and $\Zp := \mbN \setminus \{0\}$. For the rest of this paper, fix
$d,k \in \Zp, k \geq 2$ and $\theta \in \mbR$ with $0 < \theta < 1$. Given $E,F \in \fin(\ome{k})$,
we write $E < F$ (resp. $E \leq F$) to denote that $\maxprec(E) \prec \minprec(F)$ (resp.
$\maxprec(E) \preceq \minprec(F)$), and in this case we say that $E$ and $F$ are successive.
Similarly, for $v \in \ome{k}$, we write $v < E$ (resp. $v \leq E$) whenever $\set{v} < E$ (resp.
$\set{v} \leq E$).

By $\czz(\ome{k})$ we denote the vector space of all functions $x : \ome{k} \to \mbR$ such that the
set $\supp{x} := \{ v \in \ome{k} : x(v) \neq 0 \}$ is finite. Usually we write $x_v$ instead of
$x(v)$. We can extend the orders defined above to vectors $x,y \in \czz(\ome{k})$: $x < y$ (resp.
$x \leq y$) iff $\supp{x} < \supp{y}$ (resp. $\supp{x} \leq \supp{y}$).

From the notation in Section \ref{secHDES} we  have  that $\ome{k} = \{ \vec{u}_1,\vec{u}_2,\ldots \}$. Denote by $(e_{\vec{u}_n})_{n=1}^{\infty}$
the canonical basis of $\czz(\ome{k})$. To simplify notation, we will usually write $e_n$ instead
of $e_{\vec{u}_n}$. So, if $x \in \czz(\ome{k})$, then $x = \sum_{n=1}^\infty x_{\vec{u}_n} 
e_{\vec{u}_n} = \sum_{n=1}^m x_{\vec{u}_n} e_{\vec{u}_n}$ for some $m \in \Zp$. Using the above
convention, we will write $x = \sum_{n=1}^\infty x_n e_n = \sum_{n=1}^m x_n e_n$. If $E \in 
\mcAR^k$, we put $Ex := \sum_{v \in E} x_v e_v$.


\subsection{Construction of $T_k(d, \theta)$}
The Banach spaces that we introduce in this section have their roots (as all subsequent
constructions \cite{Bellenot}, \cite{Schlumprecht}, \cite{ArgyrosDeliyanni1}, \cite{GowersMaurey},
\cite{ArgyrosDeliyanni2}) in Tsirelson's fundamental discovery of a reflexive Banach space $T$ with
an unconditional basis not containing $c_0$ or $\ell_p$ with $1 \leq p < \infty$ \cite{Tsirelson}.
Based on these constructions, we present the following definitions.

\bdefn \label{defAdmissible}
For $m\leq d$, we say that
$(E_i)_{i=1}^m \subset \mcAR^k$ is $d$\textit{-admissible}, or simply \textit{admissible},
if and only if $E_1 < E_2 < \cdots < E_m$.
\edefn

For $x = \sum_{n=1}^{\infty} x_n e_n \in \czz(\ome{k})$ and $j \in \mbN$, we define a non-decreasing
sequence of norms on $\czz(\ome{k})$ as follows:
\vspace{0.3cm}
\begin{flalign*}
\hspace{0.1cm} \bullet \hspace{0.1cm} \abs{x}_0 &:= \max_{n \in \Zp} \abs{x_n}, &
\end{flalign*}

\begingroup \makeatletter \def \f@size{11} \check@mathfonts
\begin{flalign*}
\hspace{0.1cm} \bullet \hspace{0.1cm} \abs{x}_{j+1} &:= \max \left \{ \abs{x}_j, \theta \max \left
\{ \sum_{i=1}^{m} \abs{E_i x}_j : 1 \leq m \leq d, (E_i)_{i=1}^{m} \; \textit{admissible}
\right \} \right \}. &
\end{flalign*}
\endgroup
For fixed $x \in \czz(\ome{k})$, the sequence $(\abs{x}_j)_{j \in \mbN}$ is bounded above by the
$\ell_1(\ome{k})$-norm of $x$. Therefore, we can set
\vspace{4pt}
\beqs
\norm{x}_{T_k(d,\theta)} := \sup_{j \in \mbN} \abs{x}_j.
\eeqs
Clearly, $\norm{\cdot}_{{T_k(d,\theta)}}$ is a norm on $\czz(\ome{k})$.

\bdefn \label{defTk}
The completion of $\czz(\ome{k})$ with respect to the norm $\norm{\cdot}_{{T_k(d,\theta)}}$ is
denoted by $({T_k(d,\theta)}, \norm{\cdot})$.
\edefn

Notice that when $k=1$, $T_1(d,\theta)$ is the space Bellenot considered \cite{Bellenot}.

For $v \in \ome{k}$ and $x \in T_k(d,\theta)$ we also
write $v < x$ whenever $v < \supp{x}$. From the preceding definition we have the following:

\bprop
If $x \in \czz(\ome{k})$ and $\abs{\supp{x}} = n$, then $\abs{x}_n = \abs{x}_{n+1} = \cdots$.
\eprop

Therefore, we conclude that for every $x \in \czz(\ome{k})$ we have
\vspace{4pt}
\beqs
\norm{x}_{T_k(d,\theta)} = \max_{j \in \mbN} \abs{x}_j.
\eeqs

The following propositions follow by standard arguments (see Proposition 2 in \cite{Schlumprecht}):

\bprop
$(e_n)_{n=1}^{\infty}$ is a 1-unconditional basis of $T_k(d,\theta)$.
\eprop

\vspace{0.1cm}
\bprop
For $x = \sum_{n=1}^{\infty} x_n e_n \in T_k(d,\theta)$ it follows that 
\vspace{6pt}
\begingroup \makeatletter \def \f@size{11} \check@mathfonts
\beqs
\norm{x} = \max \left \{ \norm{x}_\infty, \theta \sup \left \{ \sum_{i=1}^{m} \norm{E_i x} :
1 \leq m \leq d, (E_i)_{i=1}^{m} \; d\textit{-admissible} \right \} \right \},
\eeqs
\endgroup

\noindent where $\norm{x}_\infty := \sup_{n \in \Zp} \abs{x_n}$.
\eprop


\section{Subspaces of $T_k(d,\theta)$ isomorphic to $\ell_\infty^N$}\label{sec.subspaceslNinfty}

The Banach space $T_k(d,\theta)$ ``lives'' in $\ome{k}$, at the top of $\omle{k}$.  We will see that it's structure is
determined by subspaces indexed by elements in the lower branches.  
Let $s \in \omle{k}$. The tree generated by $s$ and the Banach space associated to it are
given by
\vspace{3pt}
\beqs
\tau^k[s] := \set{v \in \ome{k} : s \sqsubseteq v}
\hspace{1.3cm} \textnormal{and} \hspace{1.3cm}
T^k[s] := \Span{e_v : v \in \tau^k[s]},
\eeqs

\noindent respectively. In this section, let $N \in \Zp$ and $s_1,\ldots,s_N \in
\omle{k}$ be such that $\abs{s_1} = \cdots = \abs{s_N} < k$ and $s_1 \prec \cdots \prec s_N$. 
The following is a very useful result.

\bcor \label{cor3LemElemXvmax}
If $v \in \ome{k}$ satisfies $s_N \prec v$, then there is at most one $i \le N$ such that 
$X_v^{\max} \cap \tau^k[s_i] \neq \emptyset$.
\ecor
\bpf
Suppose that $v=(n_1,n_2,\dots,n_k)$ and let $w=(m_1,m_2,\dots,m_k)\in X^{\max}_v$.  
It follows from Lemma \ref{lemElemXvmax} that either $m_1>n_k$ or that there is $0<i<k$ such that 
$(m_1,\dots,m_i)=(n_1,\dots,n_i)$ and $m_{i+1}>n_k$. 

If $m_1>n_k$, $w$ does not belong to any $\tau^k[s_i]$ because $n_k$ is greater than any coordinate of the $s_i$'s and as a result, none of the $s_i$'s can be an initial segment of $w$.  Heence it follows from the other option of $w$ that
the the only way an $s_i$ is an initial segment of $w$ is if $s_i$ is also an initial segment of $v$.  Since all the $s_i$'s have the same length, at most one of them is an initial segment of $v$ and the result follows.
\epf

It is useful to have an analogous result to the preceding corollary but related to approximations
$E \in \mcAR^k$ instead of special maximal elements of $\mcE_k$:

\blem \label{lemEcapTau}
Suppose $E \in \mcAR^k$ and set $v := \minprec(E)$. If $s \prec v$ and $s \sqsubset v$, then $E \cap
\tau^k[s] \neq \emptyset$.
\elem

We will study the Banach space structure of the subspaces of $T_k(d,\theta)$ of the form $Z :=
T^k[s_1] \oplus T^k[s_2] \oplus \cdots \oplus T^k[s_N]$. Since $(e_{\vec{u}_n})_{n=1}^{\infty}$ is
1-unconditional, we can decompose $Z$ as $F \oplus C$, where
\vspace{3pt}

\begingroup \makeatletter \def \f@size{11} \check@mathfonts
\begin{equation} \label{eq:FC}
\begin{split}
F & = \Span{e_v \in Z : v \in \ome{k}, v \preceq s_N}
\hspace{0.3cm} \textnormal{and} \hspace{0.3cm} \\
 C & = \Span{e_v \in Z : v \in \ome{k}, s_N \prec v}.
\end{split}
\end{equation}
\endgroup

By setting $k = 2, N = 4, s_1 = (4), s_2 = (6), s_3 = (8)$, and $s_4 = (10)$, the following figure
shows the elements of $\ome{2}$ used to generate the subspaces $F$ (dashed outline) and $C$
(thicker outline) in which we decompose the subspace $T^2[(4)] \oplus T^2[(6)] \oplus
T^2[(8)] \oplus T^2[(10)]$:

\vspace{5pt}
\begin{figure}[H]
\includegraphics[scale=0.2]{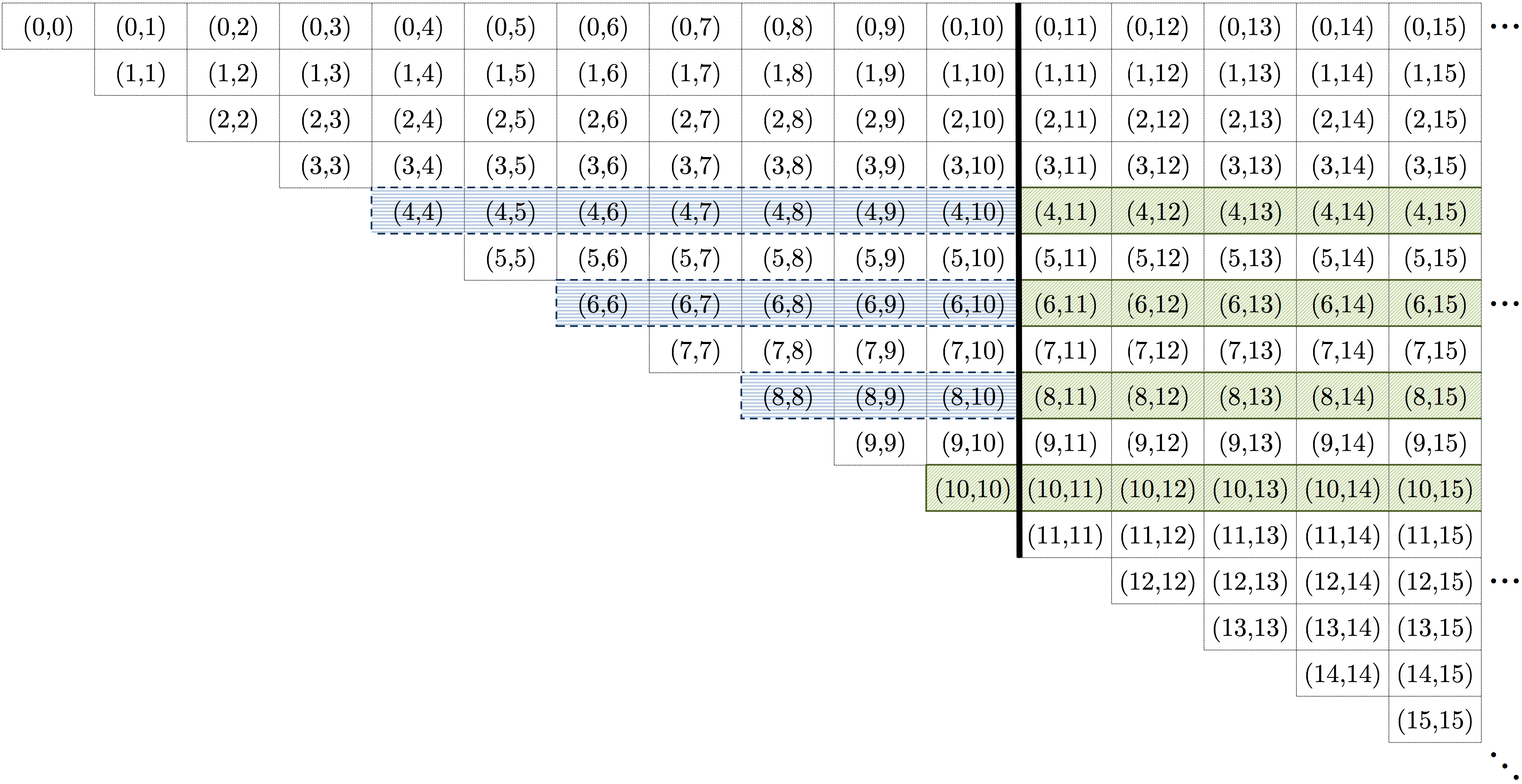}
\captionsetup{singlelinecheck=on, font=scriptsize}
\vspace{-0.2cm}
\caption{Elements of $\ome{2}$ used to generate $T^2[(4)] \oplus T^2[(6)] \oplus T^2[(8)] \oplus
T^2[(10)]$}
\end{figure}

Applying Corollary \ref{cor3LemElemXvmax} and Lemma \ref{lemEcapTau} we have:

\blem \label{lemInitialSegment}
Let $E \in \mcAR^k$ be such that $s_N \prec \minprec(E)$. If $E[T_k(d,\theta)] := \Span{e_w : w \in
E}$, then either $E[T_k(d,\theta)] \cap C = \emptyset$, or there is exactly one $i \leq N$ such
that $E[T_k(d,\theta)] \cap C \subset T^k[s_i]$.
\elem
\bpf
Suppose that $E[T_k(d,\theta)] \cap C \neq \emptyset$ and set $v := \minprec(E)$. Then, by
Corollary \ref{cor3LemElemXvmax}, there is exactly one $i \in \set{1,\ldots,N}$ such that
$X_v^{\max} \cap \tau^k[s_i] \neq \emptyset$; consequently, $s_i \sqsubset v$ and $E \cap
\tau^k[s_j] = \emptyset$ for any $j \in \set{1,\ldots,N}, j \neq i$. By hypothesis, $s_i
\preceq s_N \prec v$. Applying Lemma  \ref{lemEcapTau} we conclude that $E \cap \tau^k[s_i] \neq \emptyset$.
Hence, $E[T_k(d,\theta)] \cap C \subset T^k[s_i]$.
\epf

This lemma helps us establish the presence of arbitrarily large copies of $\ell_\infty^N$ inside
$T_k(d,\theta)$:

\bthm \label{thmSubspaceslinfty}
Suppose that $s_1\prec s_2\prec \dots \prec s_N$ belong to $\oml{k}$ and that $|s_1|=\cdots=|s_N|<k$.
Let $v\in\ome{k}$ with $s_N\prec v$ and suppose that $x\in \sum_{i=1}^N\oplus T^k[s_i]$ satifies $v<x$.  
If we decompose $x$ as $x_1+\cdots+x_N$ with $x_i \in T^k[s_i]$, then
\vspace{1pt}
\beqs
\max_{1 \leq i \leq N} \norm{x_i} \leq \norm{x} \leq \frac{\theta(d-1)}{1-\theta}
\max_{1 \leq i \leq N} \norm{x_i}.
\eeqs
In particular, if $\norm{x_1} = \cdots = \norm{x_N} = 1$, $\hbox{span}\{x_1,\dots,x_N\}$ is
isomorphic to $\ell_\infty^N$ in a canonical way and the isomorphism constant is independent of $N$ and of the $x_i$'s.
\ethm 
\bpf
Since the basis of $T_k(d,\theta)$ is unconditional,  $\max_{1 \leq i \leq N} \norm{x_i} \leq \norm{x}$. We will check the upper
bound. Let $m \in \set{1,\ldots,d}$ and $(E_i)_{i=1}^m \subset \mcAR^k$ be an admissible sequence
such that $\norm{x} = \theta \sum_{i=1}^m \norm{E_i x}$.

Without loss of generality we assume that $E_1 x \neq 0$, so that $s_N \prec \minprec(E_2)$. By Lemma
\ref{lemInitialSegment}, when $j \geq 2$, we have $E_j x = E_j x_i$ for some $i \in \set{1,\ldots,
N}$. Then it follows that $\norm{E_j x} \leq \max_{1 \leq i \leq N} \norm{x_i}$. Consequently,
\vspace{3pt}
\beqs
\norm{x} \leq \theta \norm{E_1 x} + \theta (d-1) \max_{1 \leq i \leq N} \norm{x_i}.
\eeqs

Repeat the argument for $E_1 x$. Find $m^\prime \in \set{1,\ldots,d}$ and an admissible sequence
$(F_i)_{i=1}^{m^\prime} \subset \mcAR^k$ such that $\norm{E_1 x} = \theta \sum_{i=1}^{m^\prime}
\norm{F_i(E_1 x)}$. We can assume that $F_1(E_1 x) \neq 0$, and applying Lemma \ref{lemInitialSegment}
once again we conclude that for $j \geq 2, \norm{F_j(E_1 x)} \leq \max_{1 \leq i \leq N} \norm{x_i}$.
Then,
\vspace{3pt}
\beqs
\norm{x} \leq \theta \left( \theta \norm{F_1(E_1 x)} + \theta (d-1) \max_{1 \leq i \leq N}
         \norm{x_i} \right) + \theta (d-1) \max_{1 \leq i \leq N} \norm{x_i}.           
\eeqs


Iterating this process we conclude that
\vspace{3pt}
\beqs
\norm{x} \leq \sum_{n=1}^\infty \theta^n (d-1) \max_{1 \leq i \leq N} \norm{x_i}
         \leq \frac{\theta (d-1)}{1-\theta} \max_{1 \leq i \leq N} \norm{x_i}.    
\eeqs
\epf


\section{Block Subspaces of $T_k(d,\theta)$ isomorphic to $\ell_p$}\label{sec.blocksubspaces}

For the rest of this paper suppose that $d \theta > 1$ and let $p \in (1,\infty)$ be determined by
the equation $d\theta = d^{1/p}$. Bellenot proved that $T_1(d,\theta)$ is isomorphic to $\ell_p$
(see Theorem \ref{thmLCH}). The same result was then proved by Argyros and Deliyanni in
\cite{ArgyrosDeliyanni1} with different arguments which can be extended to more general cases like
ours. In this section we show that we can find many copies of $\ell_p$ spaces inside $T_k(d,\theta)$
for $k \ge 2$:

\bthm \label{thmBlockSubspaceslp}
Suppose that $(x_i)_{i=1}^\infty$ is a normalized block sequence in $T_k(d,\theta)$ and that we
can find a sequence $(v_i)_{i=1}^\infty \subset \ome{k}$ such that:
\begin{enumerate}
\item 
$v_1\leq x_1<v_2\leq x_2<v_3\leq x_3<v_4\leq x_4<\cdots$

\item
$\supp{x_i} \subset X_{v_i}^{\max}$ and $v_{i+1} \in X_{v_i}^{\max}$ for every $i\geq1$.

\end{enumerate}
Then, $(x_i)$ is equivalent to the basis of $\ell_p$.
\ethm

Notice that Corollary \ref{cor1LemElemXvmax} implies that for every $j\geq i$, $v_j\in X_{v_i}^{\max}$ and $\supp{x_j}\subset X_{v_i}^{\max}$.
Theorem \ref{thmBlockSubspaceslp} allows us to identify natural subspaces of $T_k(d,\theta)$ isomorphic to $\ell_p$.  For example, it implies
that the top trees of $T_k(d,\theta)$ are isomorphic to $\ell_p$.  
In section \ref{sec.embedding} we will see that the top trees are isometrically isomorphic to $T_1(d,\theta)$.

\bcor \label{corLastTreelp}
If $s \in \omle{k}$ and $\abs{s} = k-1$, then $T^k[s]$ is isomorphic to $\ell_p$.
\ecor

\begin{proof}
Suppose that $s=(s_1,\dots,c_{k-1})$.  Define $v_1=(s_1,\dots,c_{k-1},c_{k-1})$, $v_2=(s_1,\dots,c_{k-1},c_{k-1}+1)$, 
$v_3=(s_1,\dots,c_{k-1},c_{k-1}+2),\cdots$ and let $x_i=e_{v_i}$.  If follows from Lemma \ref{lemElemXvmax}
that the $(v_i)$'s and $(x_i)$'s satisfy the hypothesis of Theorem \ref{thmBlockSubspaceslp} and the result follows.
\end{proof}

The ``diagonal'' subspaces of $T_k(d,\theta)$ are also isomorphic to $\ell_p$ spaces.  For $s= (s_1,\dots,c_l) \in \omle{k}$ with $ l<k$, define 
\begin{align*}
v_1 &= (s_1,\dots,c_l,c_l,\dots,c_l)\\
v_2 &= (s_1,\dots,c_l,c_l+1,\dots,c_l+1)\\
v_3 &= (s_1,\dots,c_l,c_l+2,\dots,c_l+2)\\
\vdots & \hspace{3cm} \vdots
\end{align*}
and $x_i=e_{v_i}$ for $i\geq1$.  Then we use Lemma \ref{lemElemXvmax} to verify that the
$(v_i)$'s and $(x_i)$'s satisfy the hypothesis of Theorem \ref{thmBlockSubspaceslp} and we conclude

\begin{cor}
$D[s]=\overline{span}\{e_{v_i}:i\geq1\}$ is isomorphic to $\ell_p$, for every $s \in \oml{k}$
\end{cor}

\vspace{5pt}
We prove Theorem \ref{thmBlockSubspaceslp} in two steps.  First we prove the lower $\ell_p$-estimate using Bellenot's space  
$T_1(d,\theta)$.  Then we prove the upper $\ell_p$-estimate in a more general case. 
Denote by $(t_i)$ the canonical basis of $T_1(d, \theta)$. In order to avoid confusion, we will
write $\norm{\cdot}_1$ to denote the norm on $T_1(d, \theta)$.

\bprop \label{propLowerlpEstimate}
Under the same hypotheses of Theorem \ref{thmBlockSubspaceslp}, for a finitely supported 
$z=\sum_ia_ix_i\in T_k(d,\theta)$, we have:
\vspace{4pt}
\beqs
\frac{1}{2d} \left( \sum\nolimits_i \abs{a_i}^p \right)^{1/p} \leq \norm{\sum\nolimits_i a_i t_i}_1
\leq \norm{\sum\nolimits_i a_i x_i}_{T_k(d,\theta)}.
\eeqs
\eprop
\bpf

Let $x = \sum_i a_i t_i\in T_1(d,\theta)$. Following Bellenot \cite{Bellenot}, either $\norm{x}_1 = \max_i |a_i|$,
or there exist $m \in \set{1,\ldots,d}$ and $E_1 < E_2 < \cdots < E_m$ such that $\norm{x}_1 =
\sum_{j=1}^m \theta \norm{E_j x}_1$. For each $j \in \set{1,\ldots,m}$, either $\norm{E_j x}_1 =
\max_i \{|a_i| : i \in E_j\}$, or there exist $m^\p \in \set{1,\ldots,d}$ and $E_{j1} < E_{j2} <
\cdots < E_{jm^\p}$ subsets of $E_j$ such that $\norm{E_j x}_1 = \sum_{l=1}^{m^\p} \theta
\norm{E_{jl} x}_1$. Since the sequence $(a_i)$ has only finitely many non-zero terms, this process
ends and $x$ is normed by a tree.

We will prove the result by induction on the height of the tree. If $\norm{x}_1 =
\max_i |a_i|$, the result follows.  Since the basis of $T_k(d,\theta)$ is unconditional and the $(x_i)$'s are a normalized block basis of $T_k(d,\theta)$,
we have that $\max_i |a_i| \leq \norm{\sum_i a_i x_{i}}$. 

Suppose that the result is proved for elements of $T_1(d,\theta)$ that are normed by trees of
height less than or equal to $h$ and that $x$ is normed by a tree of height $h+1$. Then, there
exist $m \in \set{1,\ldots,d}$ and $E_1 < E_2 < \cdots < E_m$ such that $\norm{x}_1 = \sum_{j=1}^m
\theta \norm{E_j x}_1$ and each $E_j x$ is normed by a tree of height less than or equal to $h$.

We will find a corresponding admissible sequence in $\mathcal{AR}^k$.
For each $j \in \set{1,\ldots,m}$, let $n_j = \min(E_j)$ and define
$$F_j=\{w\in X_{v_{n_j}}^{\max} : w\prec v_{n_{j+1}}\}\in\mathcal{AR}^k.$$
Then $F_1<F_2,\cdots<F_m$, and we conclude that $(F_j)$ is {\it d-admissible}.  Moreover we easily check from the
hypothesis of Theorem \ref{thmBlockSubspaceslp} and by Corollary \ref{cor1LemElemXvmax} that
if $i\in E_j$, then $\supp{x_i}\subset F_j$.  
Hence, using the induction hypothesis and the unconditionality of the basis of $T_k(d,\theta)$ we conclude that
\beqs
\norm{E_j x}_1 = \norm{\sum\nolimits_{l \in E_j} a_l t_l}_1
             \leq \norm{\sum\nolimits_{l \in E_j} a_l x_{l}}
	\leq\norm{ F_j z}.
\eeqs
 Therefore,
\beqs
\norm{\sum\nolimits_i a_i t_i}_1 =    \sum_{j=1}^m \theta \norm{E_j x}_1 
                                 \leq \sum_{j=1}^m \theta \norm{F_j z} 
                                 =    \norm{\sum\nolimits_i a_i x_{i}}.
\eeqs
The result follows now applying Theorem \ref{thmLCH}.
\epf

The proof of the upper bound inequality of Theorem \ref{thmBlockSubspaceslp} is harder and we need some preliminary results.

\subsection{Alternative Norm}\label{alternativeNorm}
To establish a upper $\ell_p$-estimate we will adapt an alternative and useful description of
the norm on $T_1(d,\theta)$ introduced by Argyros and Deliyanni \cite{ArgyrosDeliyanni1} to our
spaces. In that regard, the following definition plays a key role.

\bdefn \label{defAlmostAdmissible}
Let $m \in \set{1,\ldots,d}$. A sequence $(F_i)_{i=1}^m \subset \fin(\ome{k})$ is called
\textit{almost admissible} if there exists a $d$-admissible sequence $(E_n)_{n=1}^m$ in $\mathcal{AR}^k$ such
that $F_i \subseteq E_i$, for $i\leq m$.
\edefn

A standard alternative description of the norm of the space $T_k(d,\theta)$, closer to the
spirit of Tsirelson space, is as follows. Let $K_0 := \set{\pm e_i^\ast : i \in \Zp}$, and for
$n \in \mbN$,
\beqs
K_{n+1} := K_n \bigcup \set{\theta(f_1 + \cdots + f_m) : m \leq d, (f_i)_{i=1}^m \subset K_n},
\eeqs

\noindent where $(\supp{f_i})_{i=1}^m$ is almost admissible. Then, set $K := \bigcup_{n \in \mbN}
K_n$. Now, for each $n \in \mbN$ and fixed $x \in \czz(\ome{k})$, define the following
non-decreasing sequence of norms:
\beqs
\abs{x}_n^\ast := \max \set{f(x) : f \in K_n}.
\eeqs

\blem \label{lemDualNorms}
For every $n \in \mbN$ and $x \in \czz(\ome{k})$ we have $\abs{x}_n = \abs{x}_n^\ast$.
\elem
\bpf
Clearly, $\abs{x}_0 = \abs{x}_0^\ast$ for every $x \in \czz(\ome{k})$. So, let $n \in \Zp$. Suppose
$\abs{y}_j = \abs{y}_j^\ast$ for every $j \in \mbN, j < n$ and every $y \in \czz(\ome{k})$.

If $\abs{x}_n = \abs{x}_{n-1}$, then $\abs{x}_n = \abs{x}_{n-1}^\ast \leq \abs{x}_n^\ast$.
Suppose $\abs{x}_n \neq \abs{x}_{n-1}$. Let $m \in \set{1,\ldots,d}$ and $(E_i)_{i=1}^m \subset
\mcAR^k$ be an admissible sequence such that $\abs{x}_n = \theta \sum_{i=1}^m \abs{E_i x}_{n-1}$.
Then, $\abs{x}_n = \theta \sum_{i=1}^m \abs{E_i x}_{n-1}^\ast = \theta \sum_{i=1}^m f_i(E_i x)$
for some $(f_i)_{i=1}^m \subset K_{n-1}$. Define, for each $i \in \set{1,\ldots,m}$, a new
functional $f_i^\prime$ satisfying $f_i^\prime(y) = f_i(E_i y)$ for every $y \in \czz(\ome{k})$.
This implies that $\supp{f_i^\prime} = \supp{f_i^\prime} \cap E_i$. Then, $(f_i^\prime)_{i=1}^m
\subset K_{n-1}$ with $(\supp{f_i^\prime})_{i=1}^m$ almost admissible and $f_i^\prime(E_i x) =
f_i(E_i x)$. So,
\vspace{3pt}
\beqs
\theta \sum_{i=1}^m f_i(E_i x) = \theta \sum_{i=1}^m f_i^\prime(E_i x)
                               \leq \abs{E_i x}_n^\ast
                               \leq \abs{x}_n^\ast;
\eeqs

\noindent therefore, $\abs{x}_n \leq \abs{x}_n^\ast$.

Now, let $f = \theta(f_1+\cdots+f_m)$ for some $m \in \set{1,\ldots,d}$ and $(f_i)_{i=1}^m \subset
K_{n-1}$ with $(\supp{f_i})_{i=1}^m$ almost admissible. Then,
\vspace{3pt}
\beqs
f(x) = \theta \sum_{i=1}^m f_i(x) \leq \theta \sum_{i=1}^m \abs{\supp{f_i} x}_{n-1}^\ast
                                  = \theta \sum_{i=1}^m \abs{\supp{f_i} x}_{n-1}.
\eeqs

Since $(\supp{f_i})_{i=1}^m$ is almost admissible, there exists an admissible sequence
$(E_i)_{i=1}^d \subset \mcAR^k$ such that $\supp{f_i} \subseteq E_{n_i}$, where $n_1,\ldots,n_m \in
\set{1,\ldots,m}$ and $n_1 < \cdots < n_m$. So,
\beqs
\theta \sum_{i=1}^m \abs{\supp{f_i} x}_{n-1} \leq \theta \sum_{i=1}^m \abs{E_{n_i} x}_{n-1} \leq 
\abs{x}_n;
\eeqs
hence, by definition of $\abs{\cdot}_n^\ast$, we conclude that $\abs{x}_n^\ast \leq \abs{x}_n$.
\epf

Consequently, an alternative description of the norm on $T_k(d,\theta)$ is:

\bprop \label{propAlternativeNorm}
For every $x \in T_k(d,\theta)$,
\vspace{6pt}
\beqs
\norm{x} = \sup \set{f(x) : f \in K}.
\eeqs
\eprop

\subsection{Upper Bound for Theorem \ref{thmBlockSubspaceslp}}
For $m \in \set{1,\ldots,d}$ we say that $f_1,\ldots,f_m \in K$ are successive if $\supp{f_1} <
\supp{f_2} < \cdots < \supp{f_m}$.

If $f \in K$, then there exists $n \in \mbN$ such that $f \in K_n$. The ``complexity'' of
$f$ increases as $n$ increases. That is to say, for example, that the complexity of $f \in K_1$ is
less than that of $g \in K_{10}$. This is captured in the following definition.

\bdefn \label{defAnalysis}
Let $n \in \Zp$ and $\phi \in K_n \setminus K_{n-1}$. An \textit{analysis of} $\phi$ is a sequence
$(K_l(\phi))_{l=0}^n$ of subsets of $K$ such that:
\begin{enumerate}
	\item $K_l(\phi)$ consists of successive elements of $K_l$ and $\bigcup_{f \in K_l(\phi)}
	      \supp{f} = \supp{\phi}$.
				
	\item If $f \in K_{l+1}(\phi)$, then either $f \in K_l(\phi)$ or there exist $m \in \set{1,
	      \ldots,d}$ and successive $f_1,\ldots,f_m \in K_l(\phi)$ with $(\supp{f_i})_{i=1}^m$ almost
				admissible and $f = \theta(f_1+\cdots+f_m)$.
	
	\item $K_n(\phi) = \set{\phi}$.
\end{enumerate}
\edefn

Note that, by definition of the sets $K_n$, each $\phi \in K$ has an analysis. Moreover, if $f_1 \in
K_l(\phi)$ and $f_2 \in K_{l+1}(\phi)$, then either $\supp{f_1} \subseteq \supp{f_2}$ or
$\supp{f_1} \cap \supp{f_2} = \emptyset$.

Let $\phi \in K_n \setminus K_{n-1}$ and let $(K_l(\phi))_{l=0}^n$ be a fixed analysis of $\phi$.
Suppose $(x_j)_{j=1}^N$ is a finite block sequence on $T_k(d,\theta)$.

Following \cite{ArgyrosDeliyanni1}, for each $j \in \set{1,\ldots,N}$, set $l_j \in \set{0,\ldots,
n-1}$ as the smallest integer with the property that there exists at most one $g \in K_{l_j+1}(\phi)$
with $\supp{x_j} \cap \supp{g} \neq \emptyset$.

Then, define the \textit{initial part} and \textit{final part} of $x_j$ with respect to 
$(K_l(\phi))_{l=0}^n$, and denote them respectively by $x_j^\p$ and $x_j^{\p\p}$, as follows. Let
\vspace{3pt}
\beqs
\set{f \in K_{l_j}(\phi) : \supp{f} \cap \supp{x_j} \neq \emptyset} = \set{f_1,\ldots,f_m},
\eeqs

\noindent where $f_1,\ldots,f_m$ are successive. Set
\vspace{3pt}
\beqs
x_j^\p = (\supp{f_1}) x_j
\hspace{1.5cm} \text{and} \hspace{1.5cm}
x_j^{\p\p} = (\cup_{i=2}^m \, \supp{f_i}) x_j.
\eeqs

The following is a useful property of the sequence $(x_j^\p)_{j=1}^N$ (see \cite{BernuesDeliyanni}).
The analogous property is true for $(x_j^{\p\p})_{j=1}^N$.

\bprop \label{propAl}
For $l \in \set{1,\ldots,n}$ and $j \in \set{1,\ldots,N}$, set
\vspace{5pt}
\beqs
A_l(x_j^\p) := \set{f \in K_l(\phi) : \supp{f} \cap \supp{x_j^\p} \neq \emptyset}.
\eeqs
Then, there exists at most one $f \in A_l(x_j^\p)$ such that $\supp{f} \cap \supp{x_i^\p}
\neq \emptyset$ for some $i \neq j$.
\eprop
\bpf
Let $A_l(x_j^\p) = \set{f_1,\ldots,f_m}$, where $m \geq 2$ and $f_1,\ldots,f_m$ are successive.
Obviously, only $\supp{f_1}$ and $\supp{f_m}$ could intersect $\supp{x_i^\p}$ for some $i \neq j$.
We will prove that it is not possible for $f_m$.

Suppose, towards a contradiction, that $\supp{f_m} \cap \supp{x_i^\p} \neq \emptyset$ for some $i
> j$. Given that $m \geq 2$, we must have $l \leq l_j$. Consequently, there exists $g \in
K_{l_j}(\phi)$ such that $\supp{f_m} \subseteq \supp{g}$. Since $\supp{g} \cap \supp{x_j} \neq
\emptyset$ and $\supp{g} \cap \supp{x_i} \neq \emptyset$ for some $i > j$, the definition of
$x_j^{\p\p}$ implies that $\supp{g} \cap \supp{x_j} \subseteq \supp{x_j^{\p\p}}$. Therefore,
$\supp{f_m} \cap \supp{x_j^\p} = \emptyset$, a contradiction.
\epf

Following \cite{Argyros/TodorcevicBK} and \cite{BernuesDeliyanni} we now provide an upper
$\ell_p$-estimate that implies the upper $\ell_p$-estimate of Theorem \ref{thmBlockSubspaceslp}:

\bprop \label{propUpperlpEstimate}
Let $(x_j)_{j=1}^N$ be a finite normalized block basis on $T_k(d,\theta)$. Denote by 
$(t_n)_{n=1}^\infty$ the canonical basis of $T_1(d, \theta)$. Then, for any $(a_j)_{j=1}^N
\subset \mbR$, we have:
\beqs
\norm{\sum_{j=1}^N a_j x_j} \leq \frac{2}{\theta} \left( \sum_{j=1}^N \abs{a_j}^p \right)^{1/p}
\eeqs
\eprop
\bpf
In order to avoid confusion, we will write $\norm{\cdot}_1$ to denote the norm on $T_1(d, \theta)$.
By Proposition \ref{propAlternativeNorm} and Theorem \ref{thmLCH} it suffices to show that for
every $\phi \in K$,
\beqs
\abs{\phi \left(\sum_{j=1}^N a_j x_j \right)} \leq \frac{2}{\theta} \norm{\sum_{j=1}^N a_j t_j}_1.
\eeqs

By unconditionality we can assume that $x_1,\ldots,x_N$ and $\phi$ are positive. Suppose $\phi \in
K_n \setminus K_{n-1}$ for some $n \in \Zp$, and let $(K_l(\phi))_{l=0}^n$ be an analysis of
$\phi$ (see Definition \ref{defAnalysis}). Next, split each $x_j$ into its initial and final part,
$x_j^\p$ and $x_j^{\p\p}$, with respect to $(K_l(\phi))_{l=0}^n$.

We will show by induction on $l \in \set{0,1,\ldots,n}$ that for all $J \subseteq \set{1,\ldots,N}$
and all $f \in K_l(\phi)$ we have
\vspace{4pt}
\beqs
\abs{f \left(\sum_{j \in J} a_j x_j^\p \right)} \leq \frac{1}{\theta} \norm{\sum_{j \in J}
     a_j t_j}_1
\hspace{0.5cm} \text{and} \hspace{0.5cm}
\abs{f \left(\sum_{j \in J} a_j x_j^{\p\p} \right)} \leq \frac{1}{\theta} \norm{\sum_{j \in J}
     a_j t_j}_1.
\eeqs

We prove the first inequality given that the other one is analogous. Let $J \subseteq
\set{1,\ldots,N}$ and set $y = \sum_{j \in J} a_j x_j^\p$.

If $f \in K_0(\phi)$, then $f = e_i^\ast$ for some $i \in \Zp$. We want to prove that
\vspace{3pt}
\beqs
\abs{e_i^\ast(y)} \leq \frac{1}{\theta} \norm{\sum\nolimits_{j \in J} a_j t_j}_1.
\eeqs

Suppose that $e_i^\ast(y) \neq 0$. So, there exists exactly one $j_i \in J$ such that
$e_i^\ast(x_{j_i}^\p) \neq 0$. Applying Proposition \ref{propAlternativeNorm} we have
\vspace{3pt}
\beqs
\abs{e_i^\ast(y)} = \abs{e_i^\ast(a_{j_i} x_{j_i}^\p)} \leq \norm{a_{j_i} x_{j_i}^\p} \leq
\abs{a_{j_i}} \norm{x_{j_i}} \leq \norm{\sum\nolimits_{j \in J} a_j t_j}_1
\eeqs

\noindent since the basis of $T_k(d,\theta)$ is unconditional, $\norm{x_{j_i}} = 1$, and by
definition
\vspace{3pt}
\beqs
\max_{j \in J} \abs{a_j} \leq \norm{\sum\nolimits_{j \in J} a_j t_j}_1.
\eeqs

Now suppose that the desired inequality holds for any $g \in K_l(\phi)$. We will prove it for
$K_{l+1}(\phi)$. Let $f \in K_{l+1}(\phi)$ be such that $f = \theta(f_1+\cdots+f_m)$, where
$f_1,\ldots,f_m$ are successive elements in $K_l(\phi)$ with $(\supp{f_i})_{i=1}^m$ almost
admissible. Then, $1 \leq m \leq d$. Without loss of generality assume that $f_i(y) \neq 0$ for
each $i \in \set{1,\ldots,m}$. Define the following sets:
\vspace{3pt}
\begingroup \makeatletter \def \f@size{11} \check@mathfonts
\beqs
I^\p := \set{i \leq m : \exists j \in J \text{ with } f_i(x_j^\p) \neq 0
       \text{ and } \supp{f} \cap \supp{x_j^\p} \subseteq \supp{f_i}}
\eeqs
\endgroup

\noindent and
\vspace{3pt}
\beqs
J^\p := \set{j \in J : \exists i \in \set{1,\ldots,m-1} \text{ such that } f_i(x_j^\p) \neq 0
       \text{ and } f_{i+1}(x_j^\p)\neq 0}.
\eeqs

We claim that $\abs{I^\p} + \abs{J^\p} \leq m$. Indeed, if $j \in J^\p$, 
there exists $i \in \set{1,\ldots,m-1}$ such that $f_i(x_j^\p) \neq 0$ and 
$f_{i+1}(x_j^\p) \neq 0$.  From the proof of Proposition
\ref{propAl} it follows that $f_{i+1}(x_h^\p) = 0$ for every $h \neq j$, which implies that $i+1 \notin I^\p$.
Hence, we can define an injective map from $J^\p$ to $\set{1,\ldots,m} \setminus I^\p$ and we conclude that 
$\abs{I^\p} + \abs{J^\p} \leq m$.

Finally, for each $i \in I^\p$, set $D_i := \set{j \in J : \supp{f} \cap \supp{x_j^\p} \subseteq
\supp{f_i}}$. Notice that for all $i \in I^\p$ we have $D_i \cap J^\p = \emptyset$. Then,
\vspace{3pt}
\beqs
f(y) = \theta \left[\sum\nolimits_{i \in I^\p} f_i \left(\sum\nolimits_{j \in D_i} a_j x_j^\p 
       \right) + \sum\nolimits_{j \in J^\p} f(a_j x_j^\p) \right],
\eeqs
and consequently
\beqs
\abs{f(y)} \leq \theta \left[\sum\nolimits_{i \in I^\p} \abs{f_i \left(\sum\nolimits_{j \in D_i}
           a_j x_j^\p \right)} + \sum\nolimits_{j \in J^\p} \abs{f(a_j x_j^\p)} \right].
\eeqs

However, by the induction hypothesis,
\vspace{3pt}
\beqs
\abs{f_i \left(\sum\nolimits_{j \in D_i} a_j x_j^\p \right)} \leq \frac{1}{\theta} 
     \norm{\sum\nolimits_{j \in D_i} a_j t_j}_1.
\eeqs

Moreover, for each $j \in J^\p$, we have $\abs{f(a_j x_j^\p)} \leq \norm{a_j x_j^\p} \leq 
\norm{a_j x_j} = \abs{a_j} = \norm{a_j t_j}_1$. Hence,
\vspace{3pt}
\begin{align*}
\abs{f(y)} & \leq \theta \left[\frac{1}{\theta} \sum\nolimits_{i \in I^\p}
                  \norm{\sum\nolimits_{j \in D_i} a_j t_j}_1 + \frac{1}{\theta} \sum\nolimits_{
                  j \in J^\p} \norm{a_j t_j}_1 \right] \\
           & = \theta \left[\frac{1}{\theta} \sum\nolimits_{i \in I^\p} \norm{D_i \left( 
               \sum\nolimits_{j \in J} a_j t_j \right)}_1 + \frac{1}{\theta}
               \sum\nolimits_{j \in J^\p} \norm{a_j t_j}_1 \right].
\end{align*}

Given that for every $i \in I^\p, D_i \cap J^\p = \emptyset$ and $\abs{I^\p} + \abs{J^\p} \leq m
\leq d$, the family $\{D_i\}_{i \in I^\p} \cup \{\{j\}\}_{j \in J^\p}$ is $d$-admissible in $\mathcal{AR}^1$. So,
by the definition of $\norm{\cdot}_1$, we have
\vspace{8pt}

\begingroup \makeatletter \def \f@size{11} \check@mathfonts
\beqs
\begin{split}
\abs{f(y)} &  \leq \theta \left[\frac{1}{\theta} \sum\nolimits_{i \in I^\p} \norm{D_i \left(
                \sum\nolimits_{j \in J} a_j t_j \right)}_1 + \frac{1}{\theta}
                \sum\nolimits_{j \in J^\p} \norm{a_j t_j}_1 \right]                  \\
           &  \leq \frac{1}{\theta} \norm{\sum\nolimits_{j \in J} a_j t_j}_1.
\end{split}
\eeqs
\endgroup

\epf


\section{$T_k(d,\theta)$ is $\ell_p$-saturated}\label{sec.l_psat}

In this section we prove that every infinite dimensional subspace of $T_k(d,\theta)$ has a subspace
isomorphic to $\ell_p$.

Recall that the subspaces $T^k[s]$ for $s \in \omle{k}$ with $\abs{s} < k$ decompose naturally into
countable sums. Namely, if $s = (a_1, a_2, \ldots, a_l) \in \omle{k}$ and $l< k$, then
$\tau^k[s] = \bigcup_{j=a_l}^\infty \tau^k[s \upSmallFrown j]$, and therefore $T^k[s] =
\sum_{j=a_l}^\infty \oplus T^k[s \upSmallFrown j]$.

The next lemma tells us that we can find elements $v \in \tau^k[s]$ such that $X_v^{\max}$ contains
arbitrary tails of the decomposition of $\tau^k[s]$. Its proof follows from the definition of the
$\mcE_k$-tree $\widehat{X}_v$ that determines $X_v^{\max}$ (see paragraph preceding Lemma
\ref{lemElemXvmax}).

\blem \label{lemTails}
Let $s = (a_1, a_2, \ldots, a_l) \in \omle{k}$ with $l< k$. If $m \in \mbN$ with $m > a_l$
and $v = s \upSmallFrown (m, m, \ldots, m) \in \ome{k}$, then $X_v^{max} \cap \tau^k[s] =
\bigcup_{j=m}^\infty \tau^k[s \upSmallFrown j]$.
\elem

We now present the main result of this section:

\bthm
Suppose that $Z$ is an infinite dimensional subspace of $T_k(d,\theta)$. Then, there exists $Y
\subseteq Z$ isomorphic to $\ell_p$.
\ethm
\bpf
Let $Z$ be an infinite dimensional subspace of $T_k(d,\theta)$. After a standard perturbation
argument, we can assume that $Z$ has a normalized block basic sequence $(x_n)$.  

We will show that a subsequence of $(x_n)$ is isomorphic to $\ell_p$.  From Proposition
\ref{propUpperlpEstimate} we have that
\vspace{4pt}
\beqs
\norm{\sum\nolimits_n a_n x_n} \leq \frac{2}{\theta} \left( \sum\nolimits_n |a_n|^p \right)^{1/p}.
\eeqs

To obtain the lower bound we will find a subsequence and a projection $Q$ onto a subspace of the
form $T^k[s]$ such that $\left(Q\left(x_{n_j}\right)\right)$ has a lower $\ell_p$-estimate.

To this end, assume that $Z \subset T^k[s]$ for some $s \in \omle{k}$ with $\abs{s} < k$. Decompose
$T^k[s] = \sum_{j=1}^\infty \oplus T^k[s_j]$, where for each $j \in \Zp, s \sqsubset s_j,
\abs{s_j} = \abs{s}+1$, and $s_j \prec s_{j+1}$. For each $j \in \Zp$ let $Q_j: T^k[s] \to T^k[s_j]$
be the projection onto $T^k[s_j]$. Then we have the two cases:

\vspace{2mm}
\noindent \underline{Case 1:} $\forall j \in \Zp, Q_j x_n \to 0$.

\vspace{2mm}
\noindent \underline{Case 2:} $\exists j_0 \in \Zp$ such that $Q_{j_0} x_n \not\to 0$.
\vspace{2mm}

Let us look at Case 1 first.  Let $v_1$ be the first element of $\tau^k[s]$. Since there exists $p_1$
such that  $\supp{x_1} \subset \bigcup_{j=1}^{p_1} \tau^k[s_j]$, applying Lemma \ref{lemTails} we
can find $q_1 > p_1$ and $v_2 \in \tau^k[s]$ such that $v_1 \leq x_1 < v_2$ and $X_{v_2}^{\max} \cap
\tau^k[s] = \bigcup_{j=q_1}^\infty \tau^k[s_j]$. Since $Q_j x_n \to 0$ for $1 \leq j \leq q_1$ we
can find $n_2 > 1$ and $y_2 \in T^k[s]$ such that $y_2 \approx x_{n_2}$ and $Q_j y_2 = 0$ for $1
\leq j \leq q_1$.  Then we have
\vspace{4pt}
\beqs
v_1 \leq x_1 < v_2 < y_2 \ {\rm and} \ \supp{y_2} \subset X_{v_2}^{\max}.
\eeqs

We now repeat the argument. Since there exists $p_2$ such that  $\supp{y_2} \subset 
\bigcup_{j=1}^{p_2} \tau^k[s_j]$, applying Lemma \ref{lemTails} we can find $q_2 > p_2$
and $v_3 \in \tau^k[s]$ such that $v_2 < y_2 < v_3$ and $X_{v_3}^{\max} \cap \tau^k[s] = 
\bigcup_{j=q_2}^\infty \tau^k[s_j]$. Since $Q_j x_n\to 0$ for $1 \leq j \leq q_2$, we can find $n_3
> n_2$ and $y_3 \in T^k[s]$ such that $y_3 \approx x_{n_3}$ and $Q_j y_3 = 0$ for $1 \leq j \leq
q_2$. Then we have
\vspace{4pt}
\beqs
v_1 \leq x_1 < v_2 < y_2 < v_3 < y_3 \ {\rm and} \ \supp{y_2} \subset X_{v_2}^{\max}, 
                                                   \supp{y_3} \subset X_{v_3}^{\max}.
\eeqs

Proceeding this way we find a subsequence $(x_{n_i})$ and a sequence $(y_i)$ such that $y_i$ is
close enough to $x_{n_i}$. Consequently, $\Span{y_i} \approx \Span{x_{n_i}}$ and
\vspace{4pt}
\beqs 
v_1 \leq x_1 < v_2 < y_2 < v_3 < y_3 < \cdots \ {\rm and} \ \supp{y_i} \subset X_{v_i}^{\max} \ 
{\rm for} \ i>1.
\eeqs

By Proposition \ref{propLowerlpEstimate}, there exist $C_1,C_2 \in \mbR$ such that
\vspace{3pt}
\beqs
\norm{\sum\nolimits_i a_i x_{n_i}} \geq C_1 \norm{\sum\nolimits_i a_i y_{n_i}} \geq
C_2 \left( \sum\nolimits_i|a_i|^p \right)^{1/p}.
\eeqs

Let us look at Case 2 now. Find a subsequence $(n_i)$ and $\delta > 0$ such that $\delta \leq
\norm{Q_{j_0} x_{n_i}} \leq 1$.

Let $W = \Span{Q_{j_0} x_{n_i}}$. We now apply the argument in Case 1 to the sequence $Q_{j_0}
x_{n_1} < Q_{j_0} x_{n_2} < Q_{j_0} x_{n_3} < \cdots$. That is, first decompose $T^k[s_{j_0}] =
\sum_{j=1}^\infty \oplus T^k[t_j]$, where for every $j \in \Zp, s_{j_0} \sqsubset t_j, \abs{t_j} =
\abs{s_{j_0}}+1, t_j \prec t_{j+1}$. Then, look at the two cases for the sequence $\left( Q_{j_0}
x_{n_i} \right)$. If Case 1 is true, $\left( Q_{j_0} x_{n_i} \right)$ has a subsequence with a
lower $\ell_p$-estimate, and therefore $\left( x_{n_i} \right)$ has a subsequence with a lower
$\ell_p$-estimate; and if Case 2 is true, we can repeat the argument for some $t_j$ that has length
strictly larger than the length of $s_{j_0}$. If Case 1 continues to be false, after a finite
number of iterations of the same argument, the length of $t_j$ will be equal to $k-1$, and
therefore, applying Corollary \ref{corLastTreelp}, $T^k[t_j]$ would be isomorphic to $\ell_p$.
The result follows.
\epf


\section{The spaces $T_k(d,\theta)$ are not isomorphic to each other}\label{sec.notiso}

In this section we prove one of the main results of the paper

\bthm \label{isomorphismThm1}
If $k_1 \not= k_2$, then $T_{k_1}(d,\theta)$ is not isomorphic to $T_{k_2}(d,\theta)$.
\ethm

The proof goes by induction and it shows that when $k_1>k_2$, $T_{k_1}(d,\theta)$
does not embed in $T_{k_2}(d,\theta)$.  The idea is that if we had an isomorphic embedding, we
would map an $\ell_\infty^N$-sequence into an $\ell_p^N$-sequence for arbitrarily large $N$'s.
The induction step requires a stronger and more technical statement that appears in Proposition \ref{embedding_prop} below.

The proof uses the notation of the trees $\tau^k[s]$ and their Banach spaces
$T^k[s]$ (see Section \ref{sec.subspaceslNinfty}).  We start with some lemmas.  The first one is an easy consequence of the fact that the basis of 
$T_{k}(d,\theta)$ is 1-unconditional. 

\blem\label{decompose_trees}
If $s\in\omle{k}$, and $|s|<k$, there exist $s_1 \prec s_2 \prec s_3\prec\cdots$ such that $|s_i|=|s|+1$ and
$\tau^k[s]=\bigcup_{i=1}^\infty \tau^k[s_i]$.  Consequently, we decompose $T^k[s]=\sum_{i=1}^\infty\oplus T^k[s_i]$
and for $m\in\Zp$, there is a canonical projection $P_m:T^k[s]\to \sum_{i=1}^m \oplus T^k[s_i]$.
\elem

\bpf
If $s=(a_1,\dots,a_l)$, then $s_1=(a_1,\dots,a_l,a_l), s_2=(a_1,\dots,a_l,a_l+1), s_3=(a_1,\dots,a_l,a_l+2),\cdots$.
\epf

\blem\label{tree_max_approx}
Let $s\in\omle{k}$ with $|s|<k$.  Let $v=\min\tau^k[s]$.  Then $\tau^k[s]\subset X_v^{\max}$.
\elem

\bpf 
If $s=(a_1,\dots,a_l)$, then we have that $v=(a_1,\dots,a_l,a_l,\dots,a_l)$ and the result follows 
from Lemma \ref{lemElemXvmax}.
\epf

We are ready to state and prove the main proposition.

\begin{prop}\label{embedding_prop}
Let $s\in\omle{k_1}$ with $|s|<k_1$ and decompose $T^{k_1}[s]=\sum_{i=1}^\infty\oplus T^{k_1}[s_i]$
according to Lemma \ref{decompose_trees}.  Let $M\in\Zp$ and $t_1,\dots,t_M\in\omle{k_2}$ such that
$|t_1|=\cdots=|t_M|<k_2$.  

If $k_1-|s|>k_2-|t_1|$, then for every $n\in\Zp$, $\sum_{i=n}^\infty\oplus T^{k_1}[s_i]$ does
not embed into $T^{k_2}[t_1] \oplus \cdots \oplus T^{k_2}[t_M]$.
\end{prop}

\bpf
We proceed by induction. For the base case we assume that $k_2 - |t_1| = 1 < k_1 - |s|$. By
Corollary \ref{corLastTreelp}, $T^{k_2}[t_i]$ is isomorphic to $\ell_p$, and consequently so is
$T^{k_2}[t_1] \oplus \cdots \oplus T^{k_2}[t_M]$. On the other hand, Theorem
\ref{thmSubspaceslinfty} guarantees that $T^{k_1}[s]$ has arbitrarily large copies of
$\ell_\infty^N$.

\smallbreak
Suppose now that the result is true for $m\in\Zp$ and let $k_2-|t_1|= m+1 < k_1-|s|$.
\smallbreak

We will show a simpler case first, when $M=1$.  Suppose, towards a contradiction, that there exists $n\in\Zp$ and an isomorphism
\vspace{4pt}
\beqs
\Phi:\sum_{i=n}^\infty\oplus T^{k_1}[s_i] \to T^{k_2}[t_1].
\eeqs

Decompose $T^{k_2}[t_1]=\sum_{j=1}^\infty\oplus T^{k_2}[r_j]$ according to Lemma \ref{decompose_trees}.
Find $N$ large enough and $v\in\ome{k_1}$ such that $s_n\prec s_{n+1}\prec\cdots\prec s_{n+N-1}\prec v$.
We will find normalized $x_1\in T^{k_1}[s_n], x_2\in T^{k_1}[s_{n+1}],\dots, x_N\in T^{k_1}[s_{n+N-1}]$ such that 
$v < x_i$ for $i\leq N$ and we will use Theorem \ref{thmSubspaceslinfty} to conclude that 
$\hbox{span}\{x_1,\dots, x\} \approx \ell_\infty^N$. Recall that the isomorphism constant is independent of $N$
and of the $x_i$'s.

Let $v_1$ be the first element of $\tau^{k_2}[t_1]$, and let $x_1 \in T^{k_1}[s_n]$ be such that
$\norm{x_1} = 1$  and $v < x_1$. Find a finitely supported $y_1 \in T^{k_2}[t_1]$ such that $y_1
\approx\Phi(x_1)$. Applying Lemma \ref{lemTails} we can find $v_2\in\tau^{k_2}[t_1]$
such that $v_1 \leq y_1 < v_2$ and $X_{v_2}^{\max} \cap \tau^{k_2}[t_1] = \bigcup_{j=m_1+1}^\infty
\tau^{k_2}[r_j]$ for some $m_1\in\mathbb{N}$.

Since $k_1-|s_{n+1}|>k_2-|r_1|=m$ we can apply the induction hypothesis.  In particular, the map
\vspace{3pt}
\beqs
P_{m_1} \Phi_{|T^{k_1}[s_{n+1}]}: T^{k_1}[s_{n+1}] \to T^{k_2}[r_1] \oplus \cdots \oplus T^{k_2}[r_{m_1}].
\eeqs
is not an isomorphism.  As a result, there exists $x_2\in T^{k_1}[s_{n+1}]$ such that $\|x_2\|=1$ and $P_{m_1}\Phi(x_2)\approx 0$.
To add the property $v < x_2$, we decompose $T^{k_1}[s_{n+1}]=\sum_{i=1}^\infty\oplus T^{k_1}[u_i]$ 
as in Lemma \ref{decompose_trees} and apply
the induction hypothesis to $\sum_{i=p}^\infty\oplus T^{k_1}[u_i]$ for $p$ large enough.

Now that we have a normalized $x_2\in T^{k_1}[s_{n+1}]$ that satisfies $v<x_2$ and $P_m\Phi(x_2)\approx 0$, 
we find a finitely supported $y_2\in T^{k_2}[t_1]$ such that  $y_2\approx\Phi(x_2)$ and $P_{m_1} y_2=0$. Notice that $v_1 \leq y_1 < v_2 < y_2$
and that Lemma \ref{tree_max_approx} gives that $\supp{y_2}\subset X_{v_2}^{\max}$.

We now repeat the argument. Use Lemma \ref{lemTails} to find $v_3\in\tau^{k_2}[t_1]$
such that $y_2 < v_3$ and $X_{v_3}^{\max} \cap \tau^{k_2}[t_1]=\bigcup_{j=m_2+1}^\infty\tau^{k_2}[r_j].$
Then we find a normalized $x_3\in T^{k_1}[s_{n+2}]$ such that $v < x_3$ and $P_{m_2} \Phi(x_3)$ is
essentially zero. Finally, we find a finitely supported $y_3\in T^{k_2}[t_1]$ such that $y_3\approx\Phi(x_3)$ and $P_{m_2}
y_3=0$.

Proceeding this way, for every $i\leq N$, we find normalized $x_i\in T^{k_1}[s_{n+i-1}]$ with $v < x_i$,
$v_i \in \ome{k_2}$, and $y_i \in T^{k_2}[t_1]$ such that $y_i\approx\Phi(x_i)$ and
\vspace{3pt}
\beqs
v_1 \leq y_1 < v_2 < y_2 < \cdots < v_N < y_N \ {\rm and}\ v_{i+1}, \supp{y_i} \subset X_{v_i}^{\max}.
\eeqs

By Theorem \ref{thmBlockSubspaceslp},  $(y_i)_{i=1}^N$ is isomorphic to the canonical basis of $\ell_p^N$.  Hence, 
$\Phi$ maps $\ell_\infty^N$ isomorphically into $\ell_p^N$.  Since $N$
is arbitrary, this contradicts that $\Phi$ is continuous (see equation \ref{holder} below) and we conclude the case $M=1$.  

\medbreak

Let $M> 1$ and suppose, towards a contradiction, that there exists $n\in\Zp$ and an isomorphism
\vspace{4pt}
\beqs
\Phi:\sum_{i=n}^\infty\oplus T^{k_1}[s_i] \to T^{k_2}[t_1]\oplus T^{k_2}[t_2]\oplus\cdots\oplus T^{k_2}[t_M].
\eeqs

For each $j\in\Zp$ let $Q_j:\sum_{i=1}^M T^{k_2}[t_i] \to T^{k_2}[t_j]$ be the canonical projection. 
Decompose $T^{k_2}[t_j]=\sum_{i=1}^\infty T^{k_2}[r_i^j]$ as in Lemma \ref{decompose_trees} and for
each $m\in\Zp$, let $P_m^j:T^{k_2}[t_j]\to \sum_{i=1}^m T^{k_2}[r_i^j]$ be the canonical projection
onto the first $m$ blocks.  

The proof is similar to the case $M=1$.   
Find $N$ large enough and $v\in\ome{k_1}$ such that $s_n\prec s_{n+1}\prec\cdots\prec s_{n+N-1}\prec v$.
Find $x_1 \in T^{k_1}[s_n]$ such that
$\norm{x_1} = 1$  and $v < x_1$ and find a finitely supported $y_1 \in T^{k_2}[t_1]\oplus\cdots\oplus T^{k_2}[t_M]$ such that $y_1
\approx\Phi(x_1)$. 

For each $j\leq M$, let $v_1^j=\minprec(\tau^{k_2}[t_j])$.  Use Lemma \ref{lemTails} to find $v_2^j\in\tau^{k_2}[t_j]$ such that $Q_j(y_1)<v_2^j$
and $X_{v_2}^{\max}\cap\tau^{k_2}[t_j]=\bigcup_{i=m_1^j+1}^\infty \tau^{k_2}[r_i^j]$ for some $m_1^j\in\mathbb{N}$.
Let $P_1=\sum_{j=1}^M P_{m_1^j}^j$ be the projection onto the first blocks of each of the $T^{k_2}[t_j]$'s.  

Since $k_1-|s_{n+1}|>k_2-|r_1|=m$ we can apply the induction hypothesis.  In particular, the map
$P_1 \Phi_{|T^{k_1}[s_{n+1}]}$  is not an isomorphism and we can find $x_2\in T^{k_2}[s_{n+1}]$ such that
$\|x_2\|=1$ and $P_1\Phi(x_2)\approx 0$.  Arguing as in the case $M=1$, we can also assume that $v<x_2$.
We then find a finitely supported $y_2\in \sum_{j=1}^M\oplus T^{k_2}[t_j]$ such that $y_2\approx\Phi(x_2)$ and  $P_1y_2=0$.

Proceeding this way, for every $i\leq N$, we find normalized $x_i\in T^{k_1}[s_{n+i-1}]$ with $v < x_i$
and $y_i \in \sum_{j=1}^M\oplus T^{k_2}[t_j]$ such that $y_i\approx\Phi(x_i)$.  Moreover,
for every $j\leq M$, we can find $v_i^j \in \ome{k_2}$ such that
\vspace{3pt}
\beqs
v_1^j \leq Q_j(y_1) < v_2^j < Q_j(y_2) < \cdots < v_N^j < Q_j(y_N) \ {\rm and}\ v_{i+1}^j,\supp{Q_j(y_i)} \subset X_{v_i^j}^{\max}.
\eeqs

By Theorem \ref{thmBlockSubspaceslp}, there exists $C_1>0$ independent of $N$ such that for every $j\leq M$, 
$$\frac{1}{C_1} \left( \sum_{i=1}^N\|Q_j(y_i)\|^p\right)^{\frac{1}{p}}\leq 
\left\| \sum_{i=1}^N Q_j(y_i)\right \|
\leq C_1 \left( \sum_{i=1}^N\|Q_j(y_i)\|^p\right)^{\frac{1}{p}}.$$
Using the triangle inequality for $y_i=\sum_{j=1}^MQ_j(y_i)$, Holder's inequality $\biggl( \sum_{i=1}^N|a_i|
\leq N^{1/q}\biggl(\sum_{i=1}^N|a_i|^p\biggr)^{1/p}$, $\frac{1}{p}+\frac{1}{q}=1\biggr)$, 
Theorem \ref{thmBlockSubspaceslp}, and the fact that the projections $Q_j$ are contractive, we get 
\begin{align}\label{holder}
\sum_{i=1}^N\|y_i\| & \leq \sum_{i=1}^N\sum_{j=1}^M\|Q_j(y_i)\| \leq N^{1/q}\sum_{j=1}^M \left(\sum_{i=1}^N\|Q_j(y_i)\|^p\right)^{1/p} \cr
   	&\leq C_1 N^{1/q}\sum_{j=1}^M\left\|\sum_{i=1}^N Q_j(y_i)\right\|\leq C_1 N^{1/q}\sum_{j=1}^M \left\|\sum_{i=1}^N y_i\right\|   \cr
	& = C_1 N^{1/q} M \left\|\sum_{i=1}^N y_i\right\|  \approx C_1 N^{1/q} M \left\|\Phi\left( \sum_{i=1}^N x_i\right) \right\| \cr
	& \leq C_1 N^{1/q} M\|\Phi\|\left\|\sum_{i=1}^N x_i\right\|.
\end{align}

Since $N$ is arbitrary, $\sum_{i=1}^N\|y_i\|$ is of order $N$, and $\|\sum_{i=1}^N x_i\|$ stays bounded, we see that $\Phi$ cannot be bounded, contradicting
our assumption.  
\epf

\section{$T_k(d,\theta)$ embeds isometrically into $T_{k+1}(d,\theta)$}\label{sec.embedding}

For fixed $d$ and $\theta$,
the spaces $T_k(d,\theta)$, $k\ge 1$,  form a natural hierarchy in complexity over $\ell_p$.
In this section we prove that when
$j<k$, the Banach space $T_j(d,\theta)$ embeds  isomorphically into $T_k(d,\theta)$.
The basis for these results is the special  feature that 
the $j$-dimensional Ellentuck space $\mathcal{E}_j$ embeds into the $k$-dimensional Ellentuck space $\mathcal{E}_k$ in many different ways.
First, $(\omle{j},\prec)$ embeds into $(\omle{k},\prec)$ 
as  a trace above any given fixed stem of length $k-j$  in $\omle{k}$.
Second,   $(\omle{j},\prec)$  also embeds into $(\omle{k},\prec)$   as the projection of  each member in $\omle{k}$ to its first $j$ coordinates.
There are many other ways to embed 
$(\omle{j},\prec)$ into  $(\omle{k},\prec)$, and each of these embeddings will induce an embedding of $T_j(d,\theta)$ into $T_k(d,\theta)$, as
these embeddings preserve both the tree structure and  the $\prec$ order.
This  is implicit in the constructions of the  spaces $\mathcal{E}_k$  in \cite{DobrinenJSL15}
and  explicit in the recursive construction of the finite and   infinite dimensional Ellentuck spaces in  \cite{DobrinenIDE15}.

The following notation will be useful.
 Let $\Phi:\ome{k}\to\ome{k+1}$ be defined by $\Phi(v)=(0)^{\frown}v$.  One can easily check that 
$\Phi$ preserves $\prec$ and $\sqsubset$.  We can naturally extend the definition of $\Phi$ to the
finitely supported vectors of $T_k(d,\theta)$ by 
$\Phi\left(  \sum_i a_i e_{v_i} \right) =\sum_i a_i e_{\Phi(v_i)}.$

\begin{lem}\label{supperApprox}
Let $E\in\mathcal{AR}^k$ and suppose that $v=\min_{\prec}E$.  Then there exists $F\in\mathcal{AR}^{k+1}$ such that 
$\phi(E)\subset F$, $\min_{\prec}(\Phi(E))=\min_{\prec}(F)$ and
$\max_{\prec}(\Phi(E))=\max_{\prec}(F)$.
\end{lem}

\begin{proof}
Let $v=\min_{\prec}(E)$.  Proposition \ref{propXvmax} says that every $w\in E$ belongs to $X_v^{\max}$.  
Since $\Phi$ adds a $0$ at the beginning of each sequence, the characterization of Lemma \ref{lemElemXvmax} implies that
for every $w\in E$, $\Phi(w)$ belongs to $X_{\Phi(v)}^{\max}$.  Then the initial segment of $X_{\Phi(v)}^{\max}$ defined by
$F=\{s\in X_{\Phi(v)}^{\max}:s\preceq \max_{\prec}\Phi(E) \}$ satisfies all the conditions of the Lemma.
\end{proof}

\begin{cor}
Let $x\in T_k(d,\theta)$ be finitely supported.  Then $\|\Phi(x)\|_{T_{k+1}(d,\theta)}\ge \| x \|_{T_{k}(d,\theta)}$.
\end{cor}

\begin{proof}
We use induction over the length of the support of $x$.  
If $|\supp{x}|=1$ the two norms are equal.  Then we assume that the result is true for 
all vectors of $T_{k}(d,\theta)$ that have fewer than $n$ elements in their support and we  
take $x\in T_{k}(d,\theta)$ with $|\supp{x}|=n$.

If $\|x\|_{T_{k}(d,\theta)}=\|x\|_{c_0}$, the result is obviously true.  If $\|x\|_{T_{k}(d,\theta)} > \|x\|_{c_0}$, there
exist $E_1,\dots,E_d\in \mathcal{AR}^k$ such that $E_1\preceq E_2\preceq\cdots\preceq E_d$ and
$\|x\|=\theta\sum_{i=1}^d\|E_ix\|_{T_{k}(d,\theta)}.$   Notice that this implies that $|\supp{E_ix}|<n$ for
every $i\leq d$.  

By Lemma \ref{supperApprox}, there are $F_1,\dots,F_d\in\mathcal{AR}^{k+1}$ such that 
$F_1\preceq F_2\preceq\cdots\preceq F_d$ and $\Phi(E_i)\subset F_i$ for $i\leq d$.  Since
$$\|E_ix\|_{T_{k}(d,\theta)}\leq \|\Phi(E_ix)\|_{T_{k+1}(d,\theta)} = \|\Phi(E_i)\Phi(x)\|_{T_{k+1}(d,\theta)}
\leq \|F_i\Phi(x)\|_{T_{k+1}(d,\theta)}$$
we conclude that 
$$\|x\|_{T_{k}(d,\theta)}=\theta\sum_{i=1}^d\|E_ix\|_{T_{k}(d,\theta)}\leq \theta\sum_{i=1}^d\|F_i\Phi(x)\|_{T_{k+1}(d,\theta)}
\leq \|\Phi(x)\|_{T_{k+1}(d,\theta)}.$$ 
\end{proof}

To prove the reverse inequality, the following notation will be useful.
For $n\in\omega$, let  
$tr_n:\omle{k+1}\to\omle{k}$ be defined by $tr_n(n_1,n_2,\dots,n_i)=(n_2,\dots,n_i)$ if $n_1=n$ and 
$tr_n(n_1,n_2,\dots,n_i)=\emptyset$ if $n_1\not=n$.  If $E\subset\ome{k+1}$, $tr_n(E)=\{tr_n(v):v\in E, tr_n(v)\not=\emptyset\}$.
Notice that $tr_n(E)=\emptyset$ if for every $v\in E$, $tr_n(v)=\emptyset$.

\begin{lem}\label{lemReverse}
Let $x=\sum_{i=1}^N a_i e_{v_i} \in T_{k+1}(d,\theta)$.    
Given $F_1\preceq \cdots\preceq F_m$, $m\leq d$  an admissible sequence for $T_{k+1}(d,\theta)$,
there is an admissible sequence $E_1\preceq \cdots\preceq E_m$  for $T_k(d,\theta)$ such that 
each $E_i=\tr_0(F_i)$ and $\Phi(E_i x)=F_i y$.
\end{lem}

\begin{proof}
We will prove that if $F\in \mathcal{AR}^{k+1}$ and $F\Phi(x)\not=0$, then $E=tr_0(F)\in\mathcal{AR}^k$ and $\Phi(Ex)=F\Phi(x)$.
The rest of the lemma follows easily from this.

Let $F\in \mathcal{AR}^{k+1}$ with $F\Phi(x)\not=0$.  Then there exists an $\mathcal{E}_{k+1}$-tree $\widehat{X}$ such that $F$
is an initial segment of $X$.  Since $F\Phi(x)\not=0$, the first element of $F$ is of the form $(0,n_2,\dots,n_{k+1})$ for some $n_2\leq\cdots\leq n_{k+1}$.
Define $\widehat{Y}:\omle{k}\to\omle{k}$ by $\widehat{Y}(v)=tr_0\left(\widehat{X}\bigl((0)^{\frown}v\bigr)\right)$.
Since $\widehat{X}$ preserves $\prec$ and $\sqsubset$, it follows that $\widehat{Y}$ preserves those orders as well and hence $\widehat{Y}$ is an $\mathcal{E}_k$-tree.
Since $\widehat{Y}$ preserves $\prec$, it follows that $E=tr_0(F)$ is an initial segment of $Y$ (i.e., $E\in\mathcal{AR}^k$).
Since $(0)^{\frown}v\in F$ iff $v\in E$, we have
$$F\Phi(x)=\sum_{(0)^{\frown}v_i\in F} a_ie_{\Phi(v_i)}=\sum_{v_i\in E}a_ie_{\Phi(v_i)}=\Phi\left( \sum_{v_i\in E} a_i e_i\right) = \Phi(Ex) .$$

To complete the proof of the Lemma, suppose that $F_1\preceq \cdots\preceq F_m$, $m\leq d$ is an admissible sequence with $E_i=tr_0(F_i)\not=\emptyset$ for $i\leq m$.
Let $i<j$, $v\in E_i$ and $w\in E_j$.  Then $(0)^{\frown}v\in F_i$ and $(0)^{\frown}w\in F_j$, which implies that $(0)^{\frown}v\prec (0)^{\frown}w$.  And from here we conclude that $v\prec w$.  Since the elements are arbitrary we have that  $E_1\preceq \cdots\preceq E_m$.
\end{proof}

\begin{cor}
Let $x\in T_k(d,\theta)$ be finitely supported.  Then $\|\Phi(x)\|_{T_{k+1}(d,\theta)}\le \| x \|_{T_{k}(d,\theta)}$.
\end{cor}

\begin{proof}
We use induction over the length of the support of $x$.  
If $|\supp{x}|=1$ the two norms are equal.  Then we assume that the result is true for 
all vectors of $T_{k}(d,\theta)$ that have fewer than $n$ elements in their support and we  
take $x\in T_{k}(d,\theta)$ with $|\supp{x}|=n$.

If $\|\Phi(x)\|_{T_{k+1}(d,\theta)}=\|\Phi(x)\|_{c_0}$, the result is obviously true.  If $\|\Phi(x)\|_{T_{k+1}(d,\theta)} > \|\Phi(x)\|_{c_0}$, there
exist $F_1,\dots,F_m\in \mathcal{AR}^{k+1}$, $m\leq d$, such that $F_1\preceq\cdots\preceq F_m$ and
$\|\Phi(x)\|_{T_{k+1}(d,\theta)}=\theta\sum_{i=1}^m\|F_i \Phi(x)\|_{T_{k+1}(d,\theta)}.$   Notice that this implies that for $i\leq m$, $|\supp{F_i \Phi(x)}|<n$.  
Moreover, we can assume that $F_i\Phi(x)\not=0$.  

By Lemma \ref{lemReverse}, there are $E_1,\dots,E_m\in\mathcal{AR}^{k}$ such that 
$E_1\preceq\cdots\preceq E_m$ and $F_i\Phi(E_i)=\Phi(E_ix)$.  By the induction hypothesis, $\|\Phi(E_ix)\|_{T_{k+1}(d,\theta)  }\leq \|E_ix\|_{T_{k}(d,\theta)}$. Then

$$\|\Phi(x)\|_{T_{k+1}(d,\theta)}=\theta\sum_{i=1}^m\|F_i\Phi(x)\|_{T_{k+1}(d,\theta)}\leq \theta\sum_{i=1}^m\|E_i x\|_{T_{k}(d,\theta)}
\leq \| x \|_{T_{k}(d,\theta)}.$$ 
\end{proof}

Combining the previous corollaries, we obtain the following result:

\begin{thm}\label{thm.isoembd}
$\Phi:T_k(d,\theta)\to T_{k+1}(d,\theta)$ is an isometric isomorphism.
\end{thm}

Iterating this theorem, and using the notation of the beginning of Section \ref{sec.subspaceslNinfty}
we describe isometrically all subspaces of $T_k(d,\theta)$ of the form $T_k[s]$ for $s\in\omle{k}$ and $|s|<k$.

\begin{cor}\label{charSubspaces}
Suppose that $s\in\omle{k}$ with $|s|<k$.  Then $T_k[s]\subset T_k(d,\theta)$ is isometrically isomorphic to $T_{k-|s|}(d,\theta)$.
\end{cor}

\begin{rem}
Another way of embedding $T_k(d,\theta)$ into $T_{k+1}(d,\theta)$ is by sending each member $s=(s_1,\dots,s_k)\in\ome{k}$ to $\Psi(s)=(s_1,\dots,s_k,s_k)\in\ome{k+1}$.
One can check that $\Psi$ maps $T_k(d,\theta)$ isometrically into $T_{k+1}(d,\theta)$.
In fact,  for each $k_1<k_2$,
there are infinitely many different ways of embedding $\mathcal{E}_{k_1}$ into $\mathcal{E}_{k_2}$,
and  each one of these.
\end{rem}


\section{The Banach space $T(\mcA_d^k,\theta)$  }\label{sec.9}

There are 
several natural ways of constructing 
 Banach spaces
 using the Tsirelson construction and $\mathcal{AR}^k$ on high dimensional Ellentuck spaces, depending on how one interprets what the natural generalizations of the finite sets in the constructions over the Ellentuck space should be.
Finite sets can be interpreted to  be simply finite sets or they can be interpreted  as finite approximations to members of the Ellentuck space.
Thus, in   building Banach spaces on $\mathcal{E}_k$,
the admissible sets can be chosen to be either finite sets or members of $\mathcal{AR}^k$,
and likewise, the admissible sets can either be 
simply increasing 
 or required  to have a member of $\mathcal{AR}^k$ separating them.

In this section, we produce the most complex of the natural  constructions of   Banach spaces
using the Tsirelson methods
 on the high dimensional Ellentuck spaces.
We denote these spaces as $T(\mcA_d^k,\theta)$.
In $T_k(d,\theta)$ we said that $\{E_i:i\leq m\}\subset\mathcal{AR}^k$ is $d$-admissible if $m\leq d$ and
$E_1<E_2<\cdots <E_m$.  The $E_i$'s were linearly ordered, and this was indicated by the index $d$.  
In $T(\mcA_d^k,\theta)$, the order of the $E_i$'s will be determined by elements of $\mathcal{AR}_d^k$,
which leads to the definition below.

The Banach space $T(\mcA_d^k,\theta)$ shares many properties with the Banach space $T_k(d,\theta)$, their proofs being
 very similar.  In this section we indicate what parts of the proofs from previous sections for  $T_k(d,\theta)$ apply also to
$T(\mcA_d^k,\theta)$, and we provide the proofs that are different.  In particular we have the following:
(1) For $k\geq2$, $T(\mcA_d^k,\theta)$ has arbitrary large copies of $\ell_\infty^n$ with uniform isomorphism constant.
(2) $T(\mcA_d^k,\theta)$  is $\ell_p$-saturated.
(3) If $j\not= k$, then $T(\mathcal{A}_d^{j},\theta)$ and $T(\mathcal{A}_d^{k},\theta)$ are not isomorphic. 
We don't know if  $T(\mcA_d^j,\theta)$ embeds into  $T(\mathcal{A}_d^{k},\theta)$
 for $j<k$.

\bdefn \label{defAdmissibleEndpoint}
Set $\mcA_d^k := \bigcup_{i=1}^d \mcAR_i^k$ and let $m \in \set{1,2,\ldots,d}$. We say that
$(E_i)_{i=1}^m \subset \mcAR^k$ is $\mcA_d^k$\textit{-admissible} 
if and only if there exists $\set{v_1,v_2,\ldots,v_m} \in \mcA_d^k$ such that $v_1 \leq E_1 < v_2
\leq E_2 < \cdots < v_m \leq E_m$.
\edefn

\begin{figure}[H]
\includegraphics[scale=0.23]{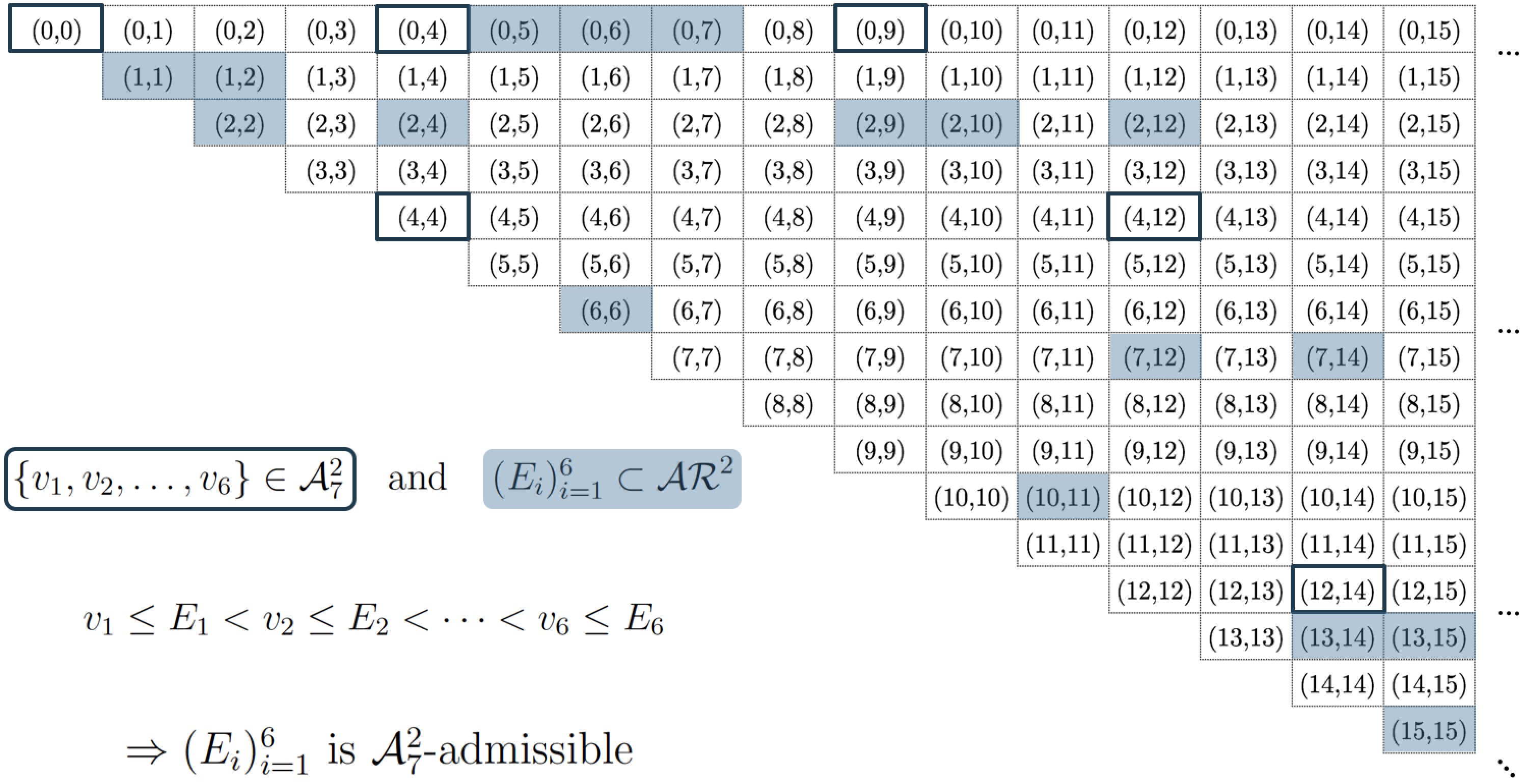}
\captionsetup{singlelinecheck=on}
\caption{An $\mathcal{A}_7^2$-admissible sequence}
\end{figure}

Following the construction of Section \ref{sec.Bconstruction}, we define a sequence of norms on $c_{00}(\ome{k})$ by 
$\abs{x}_0 := \max_{n \in \Zp} \abs{x_n}$ and 
$$ \abs{x}_{j+1} := \max \left \{ \abs{x}_j, \theta \max \left
\{ \sum_{i=1}^{m} \abs{E_i x}_j : 1 \leq m \leq d, (E_i)_{i=1}^{m} \;\mathcal{A}_d^k\textit{-admissible}
\right \} \right \},$$
and we define $T(\mcA_d^k,\theta)$ to be the completion of $c_{00}(\ome{k})$ with norm 
$$\norm{x}_{T(\mcA_d^k,\theta)}=\sup_j\abs{x}_j.$$
It follows that $T(\mcA_d^k,\theta)$ is 1-unconditional and 
$$
\norm{x}_{T(\mcA_d^k,\theta)} = \max\left\{  \norm{x}_\infty, \theta \sup \sum_{i=1}^{m} \norm{E_i x}_{T(\mcA_d^k,\theta)}\right\}
$$
where the supremum runs over all $\mathcal{A}_d^k$-admissible families $\{E_1,\dots,E_m\}$.

\vspace{3pt}

We start with the results of Section \ref{sec.subspaceslNinfty}.  
For $s\in\oml{k}$, define 
$T[\mathcal{A}_d^k, s]=\overline{\text{span}}\{e_v:v\in \tau^k[s]\}.$  Since 
Corollary \ref{cor3LemElemXvmax}, Lemma \ref{lemEcapTau}, and Lemma \ref{lemInitialSegment} depend only on
properties of $\mathcal{AR}^k$, and since these are the results used in the proof of Theorem \ref{thmSubspaceslinfty},
we obtain the following result:

\bthm \label{thmSubspaceslinfty2}
Suppose that $s_1\prec s_2\prec \dots \prec s_N$ belong to $\oml{k}$ and that $|s_1|=\cdots=|s_N|<k$.
Let $v\in\ome{k}$ with $s_N\prec v$ and suppose that $x\in \sum_{i=1}^N\oplus T[\mathcal{A}_d^k,s_i]$ satifies $v<x$.  
If we decompose $x$ as $x_1+\cdots+x_N$ with $x_i \in T[\mathcal{A}_d^k,s_i]$, then
\vspace{1pt}
\beqs
\max_{1 \leq i \leq N} \norm{x_i} \leq \norm{x} \leq \frac{\theta(d-1)}{1-\theta}
\max_{1 \leq i \leq N} \norm{x_i}.
\eeqs
In particular, if $\norm{x_1} = \cdots = \norm{x_N} = 1$, $\hbox{span}\{x_1,\dots,x_N\}$ is
isomorphic to $\ell_\infty^N$ in a canonical way and the isomorphism constant is independent of $N$ and of the $x_i$'s.
\ethm 

\vspace{3pt}

We now move to the results of Section \ref{sec.blocksubspaces}.  The main result is the same but the proof for the lower $\ell_p$-estimate 
is different.

\bthm \label{thmBlockSubspaceslp2}
Suppose that $(x_i)_{i=1}^\infty$ is a normalized block sequence in $T(\mcA_d^k,\theta)$ and that we
can find a sequence $(v_i)_{i=1}^\infty \subset \ome{k}$ such that:
\begin{enumerate}
\item 
$v_1\leq x_1<v_2\leq x_2<v_3\leq x_3<v_4\leq x_4<\cdots$

\item
$\supp{x_i} \subset X_{v_i}^{\max}$ and $v_{i+1} \in X_{v_i}^{\max}$ for every $i\geq1$.

\end{enumerate}
Then, $(x_i)$ is equivalent to the basis of $\ell_p$.
\ethm

We start with the upper $\ell_p$-estimate.  The construction of the Alternative Norm (see Subsection \ref{alternativeNorm}) is almost identical.
The main difference is the definition of almost admissible sequences.

\bdefn \label{defAlmostAdmissible}
Let $m \in \set{1,\ldots,d}$. A sequence $(F_i)_{i=1}^m \subset \fin(\ome{k})$ is called 
$\mathcal{A}_d^k$\textit{-almost admissible} if there exists an $\mcA_d^k$-admissible sequence $(E_n)_{n=1}^d$ such
that $F_i \subseteq E_{n_i}$, where $n_1,\ldots,n_m \in \set{1,\ldots,d}$ are such that $n_1 < n_2
< \cdots < n_m$.
\edefn

This results in an alternative/dual description of the norm of $T(\mcA_d^k,\theta)$.  The sets $K_n$ are smaller
than the corresponding sets for $T_k(d,\theta)$, but the proofs are identical, resulting in the upper $\ell_p$-estimate identical to
Proposition \ref{propUpperlpEstimate}.

The lower $\ell_p$-estimate is harder, because we have fewer admissible sequences than in $T_k(d,\theta)$.  We need to have 
enough ``room'' to find $\mathcal{A}_d^k$-admissible sequences, since we need to place the sets between an element of $\mathcal{A}_d^k$.
To do so, we will prove a lower $\ell_p$-estimate for the sequences $(x_{2n})$ and $(x_{2n-1})$.  Since the
closed span of $(x_{2n})$ and $(x_{2n+1})$ are complemented in the closed span of $(x_n)$, the
general result follows. We will obtain the estimate for $(x_{2n})$. The other case is similar. 

We start with the following  lemma.

\blem \label{lemPrescribedX}
Suppose that $(q_i)_{i=1}^\infty \subset \mbN$ is such that $q_1 < q_2 < \cdots$. Then, there
exists $X = \set{w_1, w_2, \ldots} \in \mcE_k$ such that for every $i \in \Zp, \max(w_i) = q_i$
and  $w_i$ is a sequence all of whose terms are in the set 
$\set{q_1, q_2, \ldots}$.
\elem

\bpf
Following Definition \ref{defE_k} we will construct inductively an $\mcE_k$-tree $\widehat{X}$
that determines $X$. 
Recall that for fixed $k\ge 2$,
the first $k$ members of $(\omle{k},\prec)$ are
$\vec{s}_0=()$, $\vec{s}_1=(0)$, and in general, for $1\le k'\le k$, $\vec{s}_{k'}$ is the sequence of $0$'s of length $k'$.
Begin by  setting  $\widehat{X}(\vec{s}_0):=()$,
and for each $1\le k'\le k$,
set  $\widehat{X}(\vec{s}_{k'})$ to be the sequence of length $k'$ with all entries being $q_1$.
Note that $\vec{u}_0$ is the sequence of $0$'s of length $k$, and that $\max (\widehat{X}(\vec{u}_0))=q_0$.

Suppose we have defined
$\widehat{X}(\vec{s}_{m'})$ for all $m'\le m$ 
so that whenever   $\vec{s}_{m'}=\vec{u}_j$ for some $j$,
then $\max(\widehat{X}(\vec{u}_j))=q_j$.
Define $\widehat{X}(\vec{s}_{m+1})$ based on the following cases:

\vspace{5pt}
\noindent \underline{Case 1:} $\abs{\vec{s}_{m+1}} = \abs{\vec{s}_m} + 1$.
Letting $ (n_1, \ldots, n_{\abs{\vec{s}_{m}}})$ denote $\widehat{X}(\vec{s}_m)$, 
set
$\widehat{X}(\vec{s}_{m+1}) := (n_1, \ldots, n_{\abs{\vec{s}_{m}}}, n_{\abs{\vec{s}_{m}}})$.

\vspace{5pt}
\noindent \underline{Case 2:}
$\abs{\vec{s}_m} = k$. 
Let $j$ be the index such that $\vec{s}_m=\vec{u}_j$.
By the induction hypothesis, 
$\max(\widehat{X}(\vec{s}_m))=q_j$.
Take $m'$ to be the index such that 
$\vec{s}_{m'}\sqsubset \vec{s}_{m+1}$
 with $|\vec{s}_{m'}|=|\vec{s}_{m+1}|-1$.
Then $\widehat{X}$ is already defined on $\vec{s}_{m'}$,
so define 
 $\widehat{X}(\vec{s}_{m+1}):={\widehat{X}(\vec{s}_{m'})}^{\frown}q_{j+1}$.
\end{proof}

\blem \label{lemEandF}
Assume that all the hypotheses of Theorem \ref{thmBlockSubspaceslp2} are satisfied.
If $m \in \set{1,\ldots,d}$ and $E_1 < E_2 < \cdots < E_m$ are finite subsets of $\Zp$, then there
exists an $\mcA_d^k$-admissible sequence $(F_i)_{i=1}^m \subset \mcAR^k$ such that
\vspace{4pt}
\beqs
E_i \subset \left\{j \in \Zp: \supp{x_{2j}} \subseteq F_i\right\}.
\eeqs
\elem

\bpf
For $i \in \set{1,\ldots,m}$, set $n_i := \min(E_i)$.  It is helpful to
keep the following picture in mind throughout this proof:
\vspace{2pt}
\beqs
\cdots \leq x_{2(n_i-1)} < v_{2n_i-1} \leq x_{2n_i-1} < v_{2n_i} \leq x_{2n_i} < \cdots \leq
x_{2(n_{i+1}-1)} < \cdots.
\eeqs

Set $q_i := \max(v_{2n_i-1})$. By hypothesis (1) of Theorem \ref{thmBlockSubspaceslp2}, it follows that 
$q_1 < q_2 < \cdots < q_m$.  Applying Lemma \ref{lemPrescribedX}, we can find
$\{w_1,w_2,\ldots,w_m\}$ in $\mcAR_d^k$ such that $\max(w_i) = q_i$ and all the terms of $w_i$ are
in $\set{q_1, q_2, \ldots, q_m}$. Consequently,
\beq \label{eqwj}
\cdots \leq x_{2(n_i-1)} < w_i \prec v_{2n_i} \leq x_{2n_i} < \cdots.
\eeq

By hypothesis, 
$X_{v_{2n_i}}^{\max}$ contains the support of $x_j$ for any $j \geq 2n_i$. Hence we define:
\beq \label{eqFj}
F_i=\{ w\in X_{v_{2n_i}}^{\max} : w\prec v_{2n_{i+1}}\}
\eeq
as the initial segment of $X_{v_{2n_i}}^{\max}$ up to (but not including) $v_{2n_{i+1}}$.  
By construction, $F_i \in \mcAR^k$, $\minprec(F_i) = v_{2n_i}$, and $F_i$ contains
the supports of 
$$x_{2n_i},x_{2(n_i+1)},x_{2(n_i+2)},\dots, x_{2(n_{i+1}-1)}.$$
Thus, from equation
(\ref{eqwj}) and (\ref{eqFj}), we have:
\vspace{3pt}
\beqs
w_1 \prec v_{2n_1} \leq F_1 < w_2 \prec v_{2n_2} \leq F_2 < \cdots \leq F_{m-1} < w_m \prec
v_{2n_m} \leq F_m,
\eeqs 
and we conclude that $F_1,F_2,\ldots,F_m \in \mcAR^k$ is the desired $\mathcal{A}_d^k$-admissible sequence.
\epf

With this, the proof of Proposition \ref{propLowerlpEstimate} applies and we obtain the lower $\ell_p$-estimate
for the normalized block sequence $(x_{2n})$.  A similar argument gives a lower $\ell_p$-estimate for the 
normalized block sequence $(x_{2n-1})$ and we conclude the sketch of the proof of Theorem \ref{thmBlockSubspaceslp2}.

\medbreak

Since Theorems \ref{thmSubspaceslinfty2} and \ref{thmBlockSubspaceslp2} hold for $T(\mcA_d^k,\theta)$, 
we have all the elements to show that the different $T(\mcA_d^k,\theta)$'s are not isomorphic to each other.
The proof of Theorem \ref{isomorphismThm1} applies and we obtain the following:

\bthm \label{isomorphismThm2}
If $k_1 \not= k_2$, then $T(\mcA_d^{k_1},\theta)$ is not isomorphic to $T(\mcA_d^{k_2},\theta)$.
\ethm


\section{Further Directions}

In this paper, we considered two different means of constructing norms on high dimensional Ellentuck spaces.
 One required both the admissible sets and their endpoints to be finite approximations to members of $\mathcal{E}_k$, and the other only required the admissible sets to be finite approximations.
Theorems \ref{isomorphismThm1} and \ref{thm.isoembd}
show 
that this latter norm construction produces a hierarchy of Banach spaces $T_k(d,\theta)$ which embed isometrically as subspaces into Banach spaces constructed from  higher order Ellentuck spaces.

\begin{question}
For fixed $d$ and $\theta$ and $k_1<k_2$, does $T(\mathcal{A}^{k_1}_d,\theta)$ embed as an isometric subspace of $T(\mathcal{A}^{k_2}_d,\theta)$?
\end{question}

Preliminary analysis shows that if we let $d'$ be sufficiently greater than $d$,
we can show that the norm on $T(\mathcal{A}^{k_1}_d,\theta)$ is bounded by the norm on the trace subspace above $(0)$ in $T(\mathcal{A}^{k_2}_{d'},\theta)$,
where $d'$ is computed from $d$ in a straightforward manner using methods from \cite{DobrinenJSL15}.
However, we have not checked whether or not this produces and isometric subspace.

\begin{question}
For fixed $d,\theta,k$ how different are $T_k(d,\theta)$ and $T(\mathcal{A}^k_d,\theta)$?
\end{question}

We finish with some questions about the behaviors of norms over certain sequences of these spaces.

\begin{question}
What is the behavior of the norms on $T_k(\mathcal{B},\theta)$, constructed  using  barriers $\mathcal{B}$ on $\mathcal{E}_k$ of infinite rank?
\end{question}

Finally, we ask about Banach spaces constructed on infinite dimensional Ellentuck spaces from \cite{DobrinenIDE15}.

\begin{question}
What new properties of the sequence of Banach spaces emerge as we construct $T_{\al}(d,\theta)$, where $\al$ is any countable ordinal and $\mathcal{E}_{\al}$ is the $\al$-dimensional Ellentuck space?
\end{question}

As the spaces $T_k(d,\theta)$ were shown to extend the $\ell_p$ space into a natural and  very structured hierarchy,
it will be interesting to see what properties emerge in these  classes of new Banach spaces.

\end{document}